\documentclass[10pt,a4paper]{article}

\usepackage{amsthm}
\usepackage{amssymb}
\usepackage{latexsym}
\usepackage{amsmath}
\usepackage{euscript}
\usepackage{graphicx}
\usepackage{fancyhdr}
\usepackage{calc}
\usepackage{color}
\usepackage{romannum}
\usepackage{paralist}

\newtheorem{Def}{Definition}[section]
\newtheorem{Lem}[Def]{Lemma}

\newtheorem{Thm}[Def]{Theorem}
\newtheorem{Prop}[Def]{Proposition}

\def\im{{\mathop{\rm im}}}
\def\Span{{\mathop{\rm span}}}
\def\dom{{\mathop{\rm dom}}}

\def\index{{\mathop{\rm index}}}

\def\loc{{\mathop{\rm loc}}}
\def\cs{{\mathop{\rm cs}}}

\def\bff{{\mathop{\rm bf}}}
\def\tf{{\mathop{\rm tf}}}
\def\tb{{\mathop{\rm tb}}}
\def\lb{{\mathop{\rm lb}}}
\def\rb{{\mathop{\rm rb}}}

\def\eq#1{{\rm(\ref{#1})}}

\begin{document}
\pagenumbering{arabic}

\title{On the Cauchy problem for the heat equation on Riemannian manifolds with conical singularities}
\author{Tapio Behrndt}
\maketitle

\begin{abstract}
We study the existence and regularity of solutions to the Cauchy problem for the inhomogeneous heat equation on compact Riemannian manifolds with conical singularities. We introduce weighted H\"older and Sobolev spaces with discrete asymptotics and we prove existence and maximal regularity of solutions to the Cauchy problem for the inhomogeneous heat equation, when the free term lies in a weighted parabolic H\"older or Sobolev space with discrete asymptotics. This generalizes a result previously obtained by Coriasco, Schrohe, and Seiler \cite[Thm. 7.2]{CSS2} by different means.
\end{abstract}

\section{Introduction}
In this paper we study the Cauchy problem for the inhomogeneous heat equation on compact Riemannian manifolds with conical singularities. More precisely, when $(M,g)$ is a compact Riemannian manifold with conical singularities $x_1,\ldots,x_n$ and $T>0$, then we study the existence and regularity of solutions to the Cauchy problem
\begin{align}\label{HeatEquation}
\begin{split}
&\partial_tu(t,x)=\Delta_gu(t,x)+f(t,x),\quad\mbox{for }(t,x)\in(0,T)\times M',\\
&u(0,x)=0,\quad\quad\quad\quad\quad\quad\quad\quad\quad\;\mbox{for } x\in M', 
\end{split}
\end{align}
where $f:(0,T)\times M'\rightarrow\mathbb{R}$ is a given function and $M'=M\backslash\{x_1,\ldots,x_n\}$. In this problem it seems natural to assume that $f$ belongs to a weighted parabolic H\"older or Sobolev space. Weighted H\"older and Sobolev spaces are generalizations of the usual H\"older and Sobolev spaces on a compact Riemannian manifold, which have a further parameter that describes the rate of decay of a function near each conical singularity. It turns out, however, that in general the Cauchy problem \eq{HeatEquation} does not have solutions with maximal regularity in weighted parabolic H\"older or Sobolev spaces. By maximal regularity we simply mean that the solution to the Cauchy problem \eq{HeatEquation} has the most regularity, i.e. differentiability and rate of decay, one can expect. To overcome this problem we introduce weighted H\"older and Sobolev spaces with discrete asymptotics. These are enlargements of the weighted H\"older and Sobolev spaces by certain finite dimensional spaces of functions. We then prove existence and maximal regularity of solutions to the Cauchy problem \eq{HeatEquation}, when $f$ lies in a weighted parabolic H\"older or Sobolev space with discrete asymptotics. Our results make Inverse Function Theorem arguments applicable to short time existence problems for nonlinear parabolic equations on Riemannian manifolds with conical singularities. One application of this kind can be found in \cite{Behrndt} and in the author's thesis \cite{BehrndtDPhil} where the Lagrangian mean curvature flow with isolated conical singularities is studied and Theorem \ref{SobolevRegularityHeat} is applied.

We give a short overview of this paper. In \S\ref{ANALYSIS} we introduce Riemannian manifolds with conical singularities, weighted H\"older and Sobolev spaces, and we discuss the Laplace operator acting on weighted H\"older and Sobolev spaces. Then we introduce the notion of discrete asymptotics, define weighted H\"older and Sobolev spaces with discrete asymptotics, and study the Laplace operator acting on weighted H\"older and Sobolev spaces with discrete asymptotics. In \S\ref{HEATKERNEL} we first review the construction of the Friedrichs heat kernel on general Riemannian manifolds. Then we discuss in an informal way the parametrix construction for the Friedrichs heat kernel on compact Riemannian manifolds with conical singularities following Mooers \cite{Mooers}. In \S\ref{CAUCHY} we then study the Cauchy problem \eq{HeatEquation} on compact Riemannian manifold with conical singularities. We first introduce weighted parabolic H\"older and Sobolev spaces with discrete asymptotics and prove weighted Schauder and weighted $L^p$-estimates. Finally, in Theorems \ref{HoelderRegularityHeat} and \ref{SobolevRegularityHeat}, we prove existence and maximal regularity of solutions to the Cauchy problem \eq{HeatEquation} when the free term lies in a weighted parabolic H\"older or Sobolev space with discrete asymptotics. Theorems \ref{HoelderRegularityHeat} and \ref{SobolevRegularityHeat} generalize a result that was previously obtained by Coriasco, Schrohe, and Seiler in \cite[Thm. 7.2]{CSS2}.

The author wishes to thank Dominic Joyce for useful comments and J\"org Seiler for pointing out mistakes in a previous proof of Theorem \ref{SobolevRegularityHeat}. This work was supported by a Sloane Robinson Foundation Award and by an EPSRC Research studentship.

\section{Riemannian manifolds with conical singularities}\label{ANALYSIS}
We begin with the definition of Riemannian cones.

\begin{Def}\label{RiemannianCone}
Let $(\Sigma,h)$ be an $(m-1)$-dimensional compact and connected Riemannian manifold, $m\geq 1$. Let $C=(\Sigma\times(0,\infty))\sqcup\{0\}$ and $C'=\Sigma\times(0,\infty)$ and write a general point in $C'$ as $(\sigma,r)$. Define a Riemannian metric on $C'$ by $g=\mathrm dr^2+r^2h$. Then we say that $(C,g)$ is the Riemannian cone over $(\Sigma,h)$ with Riemannian cone metric $g$.
\end{Def}

Next we define compact Riemannian manifolds with conical singularities.

\begin{Def}\label{DefinitionConicalSingularities}
Let $(M,d)$ be a metric space, $x_1,\ldots,x_n$ distinct points in $M$, and denote $M'=M\backslash\{x_1,\ldots,x_n\}$. Assume that $M'$ has the structure of a smooth and connected $m$-dimensional manifold, and that we are given a Riemannian metric $g$ on $M'$ that induces the metric $d$ on $M'$. Then we say that $(M,g)$ is an $m$-dimensional Riemannian manifold with conical singularities $x_1,\ldots,x_n$, if the following hold.\vspace{0.17cm}

\begin{compactenum}
\item[{\rm(i)}] We are given $R>0$ such that $d(x_i,x_j)>2R$ for $1\leq i<j\leq n$ and compact and connected $(m-1)$-dimensional Riemannian manifolds $(\Sigma_i,h_i)$ for $i=1,\ldots,n$. Denote by $(C_i,g_i)$ the Riemannian cone over $(\Sigma_i,h_i)$ for $i=1,\ldots,n$. 
\item[{\rm(ii)}] For $i=1,\ldots,n$ denote $S_i=\{x\in M\;:\;0<d(x,x_i)<R\}$. Then there exist $\mu_i\in\mathbb{R}$ with $\mu_i>0$ and diffeomorphisms $\phi_i:\Sigma_i\times(0,R)\rightarrow S_i$, such that
\begin{equation}\label{ConicalSingularity}
\left|\nabla^k(\phi_i^*(g)-g_i)\right|=O(r^{\mu_i-k})\quad\mbox{as }r\longrightarrow 0\;\mbox{for }k\in\mathbb{N}
\end{equation}
and $i=1,\ldots,n$. Here $\nabla$ and $|\cdot|$ are computed using the Riemannian cone metric $g_i$ on $\Sigma_i\times(0,R)$ for $i=1,\ldots,n$.\vspace{0.17cm}
\end{compactenum}
Additionally, if $(M,d)$ is a compact metric space, then we say that $(M,g)$ is a compact Riemannian manifold with conical singularities. 
\end{Def}

Finally we introduce the notion of a radius function.

\begin{Def}
Let $(M,g)$ be a Riemannian manifold with conical singularities as in Definition \ref{DefinitionConicalSingularities}. A radius function on $M'$ is a smooth function $\rho:M'\rightarrow(0,1]$, such that $\rho\equiv 1$ on $M'\backslash\bigcup_{i=1}^nS_i$ and
\begin{equation}\label{RadiusFunctionDecay}
|\phi_i^*(\rho)-r|=O(r^{1+\varepsilon})\quad\mbox{as }r\longrightarrow 0
\end{equation}
for some $\varepsilon>0$. Here $|\cdot|$ is computed using the Riemannian cone metric $g_i$ on $\Sigma_i\times(0,R)$ for $i=1,\ldots,n$. A radius function always exists.
\end{Def}

If $\rho$ is a radius function on $M'$ and $\boldsymbol\gamma=(\gamma_1,\ldots,\gamma_n)\in\mathbb{R}^n$, then we define a function $\rho^{\boldsymbol\gamma}$ on $M'$ as follows. On $S_i$ we set $\rho^{\boldsymbol\gamma}=\rho^{\gamma_i}$ for $i=1,\ldots,n$ and $\rho^{\boldsymbol\gamma}\equiv 1$ otherwise. Moreover, if $\boldsymbol\gamma,\boldsymbol\mu\in\mathbb{R}^n$, then we write $\boldsymbol\gamma\leq\boldsymbol\mu$ if $\gamma_i\leq\mu_i$ for $i=1,\ldots,n$, and $\boldsymbol\gamma<\boldsymbol\mu$ if $\gamma_i<\mu_i$ for $i=1,\ldots,n$. Finally, if $\boldsymbol\gamma\in\mathbb{R}^n$ and $a\in\mathbb{R}$, then we denote $\boldsymbol\gamma+a=(\gamma_1+a,\ldots,\gamma_n+a)\in\mathbb{R}^n$.

\subsection{Weighted H\"older and Sobolev spaces}\label{WeightedSpaces}
Throughout this subsection $(M,g)$ will denote a compact $m$-dimensional Riemannian manifold with conical singularities as in Definition \ref{DefinitionConicalSingularities}.

We begin by introducing weighted $C^k$-spaces and weighted H\"older spaces. For $k\in\mathbb{N}$ we denote by $C^k_{\loc}(M)$ the space of $k$-times continuously differentiable functions $u:M'\rightarrow\mathbb{R}$ and we set $C^{\infty}(M')=\bigcap_{k\in\mathbb{N}}C^k_{\loc}(M')$, which is the space of smooth functions on $M'$. For $\boldsymbol\gamma\in\mathbb{R}^n$ we define the $C^k_{\boldsymbol\gamma}$-norm by
\[
\|u\|_{C^k_{\boldsymbol\gamma}}=\sum_{j=0}^k\sup_{x\in M'}|\rho(x)^{-\boldsymbol\gamma+j}\nabla^ju(x)|\quad\mbox{for }u\in C^k_{\loc}(M'),
\]
whenever it is finite. A different choice of radius function defines an equivalent norm. Note that $u\in C^k_{\loc}(M')$ has finite $C^k_{\boldsymbol\gamma}$-norm if and only if $\nabla^ju$ grows at most like $\rho^{\boldsymbol\gamma-j}$ for $j=0,\ldots,k$ as $\rho\rightarrow 0$. We define the weighted $C^k$-space $C^k_{\boldsymbol\gamma}(M')$ by
\[
C^k_{\boldsymbol\gamma}(M')=\bigl\{u\in C^k_{\loc}(M')\;:\;\|u\|_{C^k_{\boldsymbol\gamma}}<\infty\bigr\}.
\]
Then $C^k_{\boldsymbol\gamma}(M')$ is a Banach space. We also set $C^{\infty}_{\boldsymbol\gamma}(M')=\bigcap_{k\in\mathbb{N}}C^k_{\boldsymbol\gamma}(M')$. The space $C^{\infty}_{\boldsymbol\gamma}(M')$ is in general not a Banach space.

Next we introduce weighted H\"older spaces. Let $\alpha\in(0,1)$ and $T$ be a tensor field over $M'$. We define a seminorm by
\[
[T]_{\alpha,\boldsymbol\gamma}=\sup_{\substack{x\neq y\in M'\\ d(x,y)<\delta_g(x)}}\left\{\min\left\{\rho(x)^{-\boldsymbol\gamma},\rho(y)^{-\boldsymbol\gamma}\right\}\frac{|T(x)-T(y)|}{d(x,y)^{\alpha}}\right\},
\]
whenever it is finite. Here $d(x,y)$ denotes the Riemannian distance of $x$ and $y$ with respect to $g$, and $\delta_g(x)$ denotes the injectivity radius of $g$ at $x$. Moreover, $|T(x)-T(y)|$ is understood in the sense that we first take the parallel transport of $T(x)$ along the unique minimizing geodesic connecting $x$ and $y$, and then compute the norm at the point $y$. We define the $C^{k,\alpha}_{\boldsymbol\gamma}$-norm by
\[
\|u\|_{C^{k,\alpha}_{\boldsymbol\gamma}}=\|u\|_{C^k_{\boldsymbol\gamma}}+[\nabla^ku]_{\alpha,\boldsymbol\gamma-k}\quad\mbox{for }u\in C^{k}_{\loc}(M'),
\]
whenever it is finite. The weighted H\"older space $C^{k,\alpha}_{\boldsymbol\gamma}(M')$ is given by
\[
C^{k,\alpha}_{\boldsymbol\gamma}(M')=\left\{u\in C^{k}_{\boldsymbol\gamma}(M')\;:\;\|u\|_{C^{k,\alpha}_{\boldsymbol\gamma}}<\infty\right\}.
\]
Then $C^{k,\alpha}_{\boldsymbol\gamma}(M')$ is a Banach space.

Next we define Sobolev spaces on $M'$. For a $k$-times weakly differentiable function $u:M'\rightarrow\mathbb{R}$ the $W^{k,p}$-norm is given by
\[
\|u\|_{W^{k,p}}=\left(\sum_{j=0}^k\int_{M'}|\nabla^ju|^p\;\mathrm dV_g\right)^{1/p},
\] 
whenever it is finite. Denote by $W^{k,p}_{\loc}(M')$ the space of $k$-times weakly differentiable functions on $M'$ that have locally a finite $W^{k,p}$-norm and define the Sobolev space $W^{k,p}(M')$ by
\[
W^{k,p}(M')=\left\{u\in W^{k,p}_{\loc}(M')\;:\;\|u\|_{W^{k,p}}<\infty\right\}.
\]
Then $W^{k,p}(M')$ is a Banach space. If $k=0$, then we write $L^p_{\loc}(M')$ and $L^p(M')$ instead of $W^{0,p}_{\loc}(M')$ and $W^{0,p}(M')$, respectively. Moreover, if $p=2$ we can define a scalar product on $W^{k,2}(M')$ by
\begin{equation}\label{ScalarProduct}
\langle u,v\rangle_{W^{k,2}}=\sum_{j=0}^k\int_{M'}g(\nabla^ju,\nabla^jv)\;\mathrm dV_g\quad\mbox{for }u,v\in W^{k,2}(M').
\end{equation}
Thus $W^{k,2}(M')$ is a Hilbert space.

Finally we define weighted Sobolev spaces. For $k\in\mathbb{N}$, $p\in[1,\infty)$, and $\boldsymbol\gamma\in\mathbb{R}^n$ we define the $W^{k,p}_{\boldsymbol\gamma}$-norm by
\[
\|u\|_{W^{k,p}_{\boldsymbol\gamma}}=\left(\sum_{j=0}^k\int_{M'}|\rho^{-\boldsymbol\gamma+j}\nabla^ju|^p\rho^{-m}\;\mathrm dV_g\right)^{1/p}\quad\mbox{for }u\in W^{k,p}_{\loc}(M'),
\]
whenever it is finite. A different choice of radius function defines an equivalent norm. We define the weighted Sobolev space $W^{k,p}_{\boldsymbol\gamma}(M')$ by
\[
W^{k,p}_{\boldsymbol\gamma}(M')=\left\{u\in W^{k,p}_{\loc}(M')\;:\;\|u\|_{W^{k,p}_{\boldsymbol\gamma}}<\infty\right\}.
\]
Then $W^{k,p}_{\boldsymbol\gamma}(M')$ is a Banach space. If $k=0$, then we write $L^p_{\boldsymbol\gamma}(M')$ instead of $W^{0,p}_{\boldsymbol\gamma}(M')$. Note that $L^p(M')=L^p_{-m/p}(M')$ and that $C^{\infty}_{\cs}(M')$, the space of smooth functions on $M'$ with compact support, is dense in $W^{k,p}_{\boldsymbol\gamma}(M')$ for every $k\in\mathbb{N}$, $p\in[1,\infty)$, and $\boldsymbol\gamma\in\mathbb{R}^n$. Moreover if $p=2$, then we can define a scalar product by
\[
\langle u,v\rangle_{W^{k,2}_{\boldsymbol\gamma}}=\sum_{j=0}^k\int_{M'}\rho^{-2\boldsymbol\gamma+2j}g(\nabla^ju,\nabla^jv)\rho^{-m}\;\mathrm dV_g\quad\mbox{for }u,v\in W^{k,2}_{\boldsymbol\gamma}(M').
\]
Thus $W^{k,2}_{\boldsymbol\gamma}(M')$ is a Hilbert space.

The following proposition follows immediately from H\"older's inequality.

\begin{Prop}\label{Duality}
Let $p,q\in(1,\infty)$ with $\frac{1}{p}+\frac{1}{q}=1$ and $\boldsymbol\gamma\in\mathbb{R}^n$. Then $\langle\cdot,\cdot\rangle_{L^2}$ as defined in \eq{ScalarProduct} defines a dual pairing $L^{p}_{\boldsymbol\gamma}(M')\times L^q_{-m-\boldsymbol\gamma}(M')\rightarrow\mathbb{R}$. In particular $L^{p}_{\boldsymbol\gamma}(M')$ and $L^q_{-m-\boldsymbol\gamma}(M')$ are Banach space duals of each other.
\end{Prop}

An important tool in the study of partial differential equations is the Sobolev Embedding Theorem, which gives embeddings between Sobolev spaces and of Sobolev spaces into H\"older spaces. The next theorem is a version of the Sobolev Embedding Theorem for weighted H\"older and Sobolev spaces.

\begin{Thm}\label{WeightedSobolevEmbedding}
Let $(M,g)$ be a compact $m$-dimensional Riemannian manifold with conical singularities as in Definition \ref{DefinitionConicalSingularities}. Let $k,l\in\mathbb{N}$, $p,q\in[1,\infty)$, $\alpha\in(0,1)$, and $\boldsymbol\gamma,\boldsymbol\delta\in\mathbb{R}^n$. Then the following hold.\vspace{0.17cm}
\begin{compactenum}
\item[{\rm(i)}] If $\frac{1}{p}\leq\frac{1}{q}+\frac{k-l}{m}$ and $\boldsymbol\gamma\geq\boldsymbol\delta$ then $W^{k,p}_{\boldsymbol\gamma}(M')$ embeds continuously into $W^{l,q}_{\boldsymbol\delta}(M')$ by inclusion.
\item[{\rm(ii)}] If $k-\frac{m}{p}\geq l+\alpha$ and $\boldsymbol\gamma\geq\boldsymbol\delta$, then $W^{k,p}_{\boldsymbol\gamma}(M')$ embeds continuously into $C^{l,\alpha}_{\boldsymbol\delta}(M')$ by inclusion.
\end{compactenum}
\end{Thm}

\noindent The proof of the Theorem \ref{WeightedSobolevEmbedding} can be found in Bartnik \cite[Thm. 1.2]{Bartnik} for the case when $(M,g)$ is an asymptotically Euclidean manifold. The proof of the Sobolev Embedding Theorem for weighted spaces on compact Riemannian manifolds with conical singularities is then a simple modification of Bartnik's proof.

\subsection{The Laplace operator on compact Riemannian manifolds with conical singularities. I}
In this subsection we study the Laplace operator acting on weighted H\"older and Sobolev spaces. The presentation of this material mainly follow's Joyce \cite[\S 2]{Joyce}.

Let $(\Sigma,h)$ be a compact and connected $(m-1)$-dimensional Riemannian manifold, $m\geq 1$, and let $(C,g)$ be the Riemannian cone over $(\Sigma,h)$ as in Definition \ref{RiemannianCone}. A function $u:C'\rightarrow\mathbb{R}$ is said to be homogeneous of order $\alpha$, if there exists a function $\varphi:\Sigma\rightarrow\mathbb{R}$, such that $u(\sigma,r)=r^{\alpha}\varphi(\sigma)$ for $(\sigma,r)\in C'$. A straightforward computation shows that the Laplace operator on $C'$ is given by $\Delta_{g}u=\partial_r^2u+(m-1)r^{-1}\partial_ru+r^{-2}\Delta_hu$, and the following lemma is easily verified.

\begin{Lem}\label{HarmonicFunctions}
A homogeneous function $u(\sigma,r)=r^{\alpha}\varphi(\sigma)$ of order $\alpha\in\mathbb{R}$ on $C'$ with $\varphi\in C^2(\Sigma)$ is harmonic if and only if $\Delta_h\varphi=-\alpha(\alpha+m-2)\varphi$. Note that if $u$ is harmonic, then $\varphi\in C^{\infty}(\Sigma)$ by elliptic regularity.
\end{Lem}

Define
\begin{equation}\label{ExceptionalSet}
\mathcal{D}_{\Sigma}=\{\alpha\in\mathbb{R}\;:\;-\alpha(\alpha+m-2)\mbox{ is an eigenvalue of }\Delta_h\}.
\end{equation}
Then $\mathcal{D}_{\Sigma}$ is a discrete subset of $\mathbb{R}$ with no other accumulation points than $\pm\infty$. Moreover $\mathcal{D}_{\Sigma}\cap(2-m,0)=\emptyset$, since $\Delta_h$ is non-positive, and finally from Lemma \ref{HarmonicFunctions} it follows that $\mathcal{D}_{\Sigma}$ is the set of all $\alpha\in\mathbb{R}$ for which there exists a nonzero homogeneous harmonic function of order $\alpha$ on $C'$. Define a function
\[
m_{\Sigma}:\mathbb{R}\longrightarrow\mathbb{N},\quad m_{\Sigma}(\alpha)=\dim\ker(\Delta_h+\alpha(\alpha+m-2)).
\]
Then $m_\Sigma(\alpha)$ is the multiplicity of the eigenvalue $-\alpha(\alpha+m-2)$. Note that $m_{\Sigma}(\alpha)\neq 0$ if and only if $\alpha\notin\mathcal{D}_{\Sigma}$. Finally we define a function $M_{\Sigma}:\mathbb{R}\rightarrow\mathbb{Z}$ by
\begin{equation*}
M_{\Sigma}(\delta)=-\sum_{\alpha\in\mathcal{D}_{\Sigma}\cap(\delta,0)}m_{\Sigma}(\alpha)\;\mbox{if }\delta<0,\;M_{\Sigma}(\delta)=\sum_{\alpha\in\mathcal{D}_{\Sigma}\cap[0,\delta)}m_{\Sigma}(\alpha)\;\mbox{if }\delta\geq 0.
\end{equation*}
Then $M_{\Sigma}$ is a monotone increasing function that is discontinuous exactly on $\mathcal{D}_{\Sigma}$. As $\mathcal{D}_{\Sigma}\cap(2-m,0)=\emptyset$, we see that $M_{\Sigma}\equiv 0$ on $(2-m,0)$.
The set $\mathcal{D}_{\Sigma}$ and the function $M_{\Sigma}$ play an important r\^{o}le in the Fredholm theory for the Laplace operator on compact Riemannian manifolds with conical singularities, see Theorem \ref{Fredholm} below. 

The next proposition gives the weighted Schauder and $L^p$-estimates for the Laplace operator on compact Riemannian manifolds with conical singularities.

\begin{Prop}\label{SchauderLPLaplace}
Let $(M,g)$ be a compact Riemannian manifold with conical singularities as in Definition \ref{DefinitionConicalSingularities} and $\boldsymbol\gamma\in\mathbb{R}^n$. Let $u,f\in L^1_{\loc}(M')$ and assume that $\Delta_gu=f$ holds in the weak sense. Then the following hold.\vspace{0.17cm}
\begin{compactenum}
\item[{\rm(i)}] Let $k\in\mathbb{N}$ with $k\geq 2$ and $\alpha\in(0,1)$. If $f\in C^{k-2,\alpha}_{\boldsymbol\gamma-2}(M')$ and $u\in C^0_{\boldsymbol\gamma}(M')$, then $u\in C^{k,\alpha}_{\boldsymbol\gamma}(M')$. Moreover there exists a constant $c>0$ independent of $u$ and $f$, such that
\begin{equation}\label{estimate}
\|u\|_{C^{k,\alpha}_{\boldsymbol\gamma}}\leq c\left(\|f\|_{C^{k-2,\alpha}_{\boldsymbol\gamma-2}}+\|u\|_{C^0_{\boldsymbol\gamma}}\right).
\end{equation}
\item[{\rm(ii)}] Let $k\in\mathbb{N}$ with $k\geq 2$ and $p\in(1,\infty)$. If $f\in W^{k-2,p}_{\boldsymbol\gamma-2}(M')$ and $u\in L^p_{\boldsymbol\gamma}(M')$, then $u\in W^{k,p}_{\boldsymbol\gamma}(M')$. Moreover there exists a constant $c>0$ independent of $u$ and $f$, such that
\begin{equation}\label{estimate2}
\|u\|_{W^{k,p}_{\boldsymbol\gamma}}\leq c\left(\|f\|_{W^{k-2,p}_{\boldsymbol\gamma-2}}+\|u\|_{L^p_{\boldsymbol\gamma}}\right).
\end{equation}
\end{compactenum}
\end{Prop}

\noindent A proof of Proposition \ref{SchauderLPLaplace} can be found in Marshall's thesis \cite[Thm. 4.21]{Marshall}.

The next theorem is the main Fredholm theorem for the Laplace operator on compact Riemannian manifolds with conical singularities. 

\begin{Thm}\label{Fredholm}
Let $(M,g)$ be a compact $m$-dimensional Riemannian manifold with conical singularities as in Definition \ref{DefinitionConicalSingularities}, $m\geq 3$, and $\boldsymbol\gamma\in\mathbb{R}^n$. Then the following hold.\vspace{0.17cm}
\begin{compactenum}
\item[{\rm(i)}] Let $k\in\mathbb{N}$ with $k\geq 2$ and $\alpha\in(0,1)$. Then
\begin{equation}\label{p1}
\Delta_g:C^{k,\alpha}_{\boldsymbol\gamma}(M')\rightarrow C^{k-2,\alpha}_{\boldsymbol\gamma-2}(M')
\end{equation}
is a Fredholm operator if and only if $\gamma_i\notin\mathcal{D}_{\Sigma_i}$ for $i=1,\ldots,n$. If $\gamma_i\notin\mathcal{D}_{\Sigma_i}$ for $i=1,\ldots,n$, then the Fredholm index of \eq{p1} is equal to $-\sum_{i=1}^nM_{\Sigma_i}(\gamma_i)$. 
\item[{\rm(ii)}] Let $k\in\mathbb{N}$ with $k\geq 2$ and $p\in(1,\infty)$. Then
\begin{equation}\label{p2}
\Delta_g:W^{k,p}_{\boldsymbol\gamma}(M')\rightarrow W^{k-2,p}_{\boldsymbol\gamma-2}(M')
\end{equation}
is a Fredholm operator if and only if $\gamma_i\notin\mathcal{D}_{\Sigma_i}$ for $i=1,\ldots,n$. If $\gamma_i\notin\mathcal{D}_{\Sigma_i}$ for $i=1,\ldots,n$, then the Fredholm index of \eq{p2} is equal to $-\sum_{i=1}^nM_{\Sigma_i}(\gamma_i)$.\vspace{0.17cm}
\end{compactenum}
Furthermore the kernel of the operators \eq{p1} and \eq{p2} is constant in $\boldsymbol\gamma\in\mathbb{R}^n$ on the connected components of $(\mathbb{R}\backslash\mathcal{D}_{\Sigma_1})\times\cdots\times(\mathbb{R}\backslash\mathcal{D}_{\Sigma_n})$. 
\end{Thm}

\noindent The proof of Theorem \ref{Fredholm} can be found in Lockhart and McOwen \cite[Thm. 6.1]{LM} and in Marshall \cite[Thm. 6.9]{Marshall}. In fact, Lockhart and McOwen prove the second part of Theorem \ref{Fredholm} for the Laplace operator acting on weighted Sobolev spaces and Marshall deduces the first part of Theorem \ref{Fredholm} for the Laplace operator acting on weighted H\"older spaces from the results of Lockhart and McOwen.

The following proposition is a simple consequence of Proposition \ref{Duality} and Theorem \ref{Fredholm}.

\begin{Prop}\label{Cokernel}
Let $(M,g)$ be a compact $m$-dimensional Riemannian manifold with conical singularities as in Definition \ref{DefinitionConicalSingularities}, $m\geq 3$. Let $k\in\mathbb{N}$ with $k\geq 2$, $p,q\in(1,\infty)$ with $\frac{1}{p}+\frac{1}{q}=1$, and $\boldsymbol\gamma\in\mathbb{R}^n$ with $\gamma_i\notin\mathcal{D}_{\Sigma_i}$ for $i=1,\ldots,n$. Then \eq{p2} is a Fredholm operator and its cokernel is isomorphic to the kernel of the operator $\Delta_g:W^{k,q}_{2-m-\boldsymbol\gamma}(M')\rightarrow W^{k-2,q}_{-m-\boldsymbol\gamma}(M')$.
\end{Prop}

As before let $(\Sigma,h)$ be a compact and connected $(m-1)$-dimensional Riemannian manifold, $m\geq 1$, and $(C,g)$ the Riemannian cone over $(\Sigma,h)$. Define
\[
\mathcal{E}_{\Sigma}=\mathcal{D}_{\Sigma}\cup\left\{\beta\in\mathbb{R}\;:\;\beta=\alpha+2k\mbox{ for }\alpha\in\mathcal{D}_{\Sigma},\;k\in\mathbb{N}\mbox{ with }\alpha\geq 0\mbox{ and }k\geq 1\right\}
\]
and a function $n_{\Sigma}:\mathbb{R}\longrightarrow\mathbb{N}$ by
\[
n_{\Sigma}(\beta)=m_{\Sigma}(\beta)+\sum_{k\geq 1,\;2k\leq\beta}m_{\Sigma}(\beta-2k).
\]
Clearly if $\beta\notin\mathcal{E}_{\Sigma}$, then $n_{\Sigma}(\beta)=0$. Also note that if $\beta<2$, then $n_{\Sigma}(\beta)=m_{\Sigma}(\beta)$. Moreover, if $\beta\in\mathcal{E}_{\Sigma}$, then $n_{\Sigma}(\beta)$ counts the multiplicity of the eigenvalues 
\[
-\beta(\beta+m-2),-(\beta-2)((\beta-2)+m-2),\ldots,-(\beta-2k)((\beta-2k)+m-2)
\]
for $2k\leq\beta$. Finally we define a function $N_{\Sigma}:\mathbb{R}\longrightarrow\mathbb{N}$ by
\begin{equation}\label{DEFN}
N_{\Sigma}(\delta)=-\sum_{\beta\in\mathcal{D}_{\Sigma}\cap(\delta,0)}n_{\Sigma}(\beta)\;\mbox{if }\delta<0,\;N_{\Sigma}(\delta)=\sum_{\beta\in\mathcal{D}_{\Sigma}\cap[0,\delta)}n_{\Sigma}(\beta)\;\mbox{if }\delta\geq 0.
\end{equation}
Then $N_{\Sigma}(\delta)=M_{\Sigma}(\delta)$ for $\delta\leq 2$ and
\begin{equation}\label{indexxx}
M_{\Sigma}(\delta)=N_{\Sigma}(\delta)-N_{\Sigma}(\delta-2)\quad\mbox{for }\delta\in\mathbb{R}\mbox{ with }\delta>2.
\end{equation}
The set $\mathcal{E}_{\Sigma}$ and the function $N_{\Sigma}$ play a similar r\^{o}le in the study of the heat equation on compact Riemannian manifolds with conical singularities as $\mathcal{D}_{\Sigma}$ and $M_{\Sigma}$ do in the study of the Laplace operator, see Theorems \ref{HoelderRegularityHeat} and \ref{SobolevRegularityHeat} below.

\subsection{Weighted H\"older and Sobolev spaces with discrete asymptotics}
In this subsection we first explain the construction of discrete asymptotics on Riemannian manifolds with conical singularities and then define weighted H\"older and Sobolev spaces with discrete asymptotics. The notion of discrete asymptotics in our specific setting appears to be new. There is however a strong similarity between our definition of discrete asymptotics and the index sets for polyhomogeneous conormal distributions considered by Melrose \cite[Ch. 5, \S 10]{Melrose1} and the asymptotic types considered by Schulze \cite[Ch. 2, \S 3]{Schulze}. 

We first explain our motivation for the introduction of discrete asymptotics. If $(M,g)$ is a compact $m$-dimensional Riemannian manifold with conical singularities as in Definition \ref{DefinitionConicalSingularities}, $m\geq 3$, and $\boldsymbol\gamma\in\mathbb{R}^n$ with $\gamma_i\notin\mathcal{D}_{\Sigma_i}$ for $i=1,\ldots,n$, then for every $k\in\mathbb{N}$ with $k\geq 2$ and $p\in(1,\infty)$, $\Delta_g:W^{k,p}_{\boldsymbol\gamma}(M')\rightarrow W^{k-2,p}_{\boldsymbol\gamma-2}(M')$ is a Fredholm operator by Theorem \ref{Fredholm}. If $\boldsymbol\gamma>0$, then it follows from Theorem \ref{Fredholm} that the Fredholm index of $\Delta_g:W^{k,p}_{\boldsymbol\gamma}(M')\rightarrow W^{k-2,p}_{\boldsymbol\gamma-2}(M')$ is negative, so $\Delta_g:W^{k,p}_{\boldsymbol\gamma}(M')\rightarrow W^{k-2,p}_{\boldsymbol\gamma-2}(M')$ has a cokernel. The main idea behind our definition of discrete asymptotics is to enlarge the spaces $W^{k,p}_{\boldsymbol\gamma}(M')$ and $W^{k-2,p}_{\boldsymbol\gamma-2}(M')$ by finite dimensional spaces of functions that decay slower than $\rho^{\boldsymbol\gamma}$ and $\rho^{\boldsymbol\gamma-2}$, respectively, and that cancel the cokernel of $\Delta_g:W^{k,p}_{\boldsymbol\gamma}(M')\rightarrow W^{k-2,p}_{\boldsymbol\gamma-2}(M')$. More precisely our goal is to construct two finite dimensional spaces $V_1$ and $V_2$ consisting of functions that decay slower than $\rho^{\boldsymbol\gamma}$ and $\rho^{\boldsymbol\gamma-2}$, respectively, and $V_2\subset V_1$ such that the Laplace operator maps $W^{k,p}_{\boldsymbol\gamma}(M')\oplus V_1$ into $W^{k-2,p}_{\boldsymbol\gamma-2}(M')\oplus V_2$ and 
\begin{equation}\label{llkko}
\index\left\{\Delta_g:W^{k,p}_{\boldsymbol\gamma}(M')\oplus V_1\longrightarrow W^{k-2,p}_{\boldsymbol\gamma-2}(M')\oplus V_2\right\}=0
\end{equation}

We begin with the construction of the model space for the discrete asymptotics. Let $(\Sigma,h)$ be a compact and connected $(m-1)$-dimensional Riemannian manifold, $m\geq 1$, and let $(C,g)$ be the Riemannian cone over $(\Sigma,h)$. For $\gamma\in\mathbb{R}$ we denote
\[
H_{\gamma}(C')=\Span\left\{u=r^{\alpha}\varphi\;:\;0\leq\alpha<\gamma,\;\varphi\in C^{\infty}(\Sigma),\;u\mbox{ is harmonic}\right\},
\]
which is the space of homogeneous harmonic functions of order $\alpha$ with $0\leq\alpha<\gamma$. Then $\dim H_{\gamma}(C')=M_{\Sigma}(\gamma)$ for $\gamma\geq 2-m$, so $H_{\gamma}(C')$ is at least one dimensional for $\gamma>0$. We define a finite dimensional vector space $V_{\mathsf{P}_{\gamma}}(C')$ by
\[
V_{\mathsf{P}_{\gamma}}(C')=\Span\left\{v=r^{2k}u\;:\;k\in\mathbb{N},\;u=r^{\alpha}\varphi\in H_{\gamma}(C')\mbox{ and }\alpha+2k<\gamma\right\}.
\]
Note that the Laplace operator on $C'$ maps $V_{\mathsf{P}_{\gamma}}(C')\rightarrow V_{\mathsf{P}_{\gamma-2}}(C')$ for every $\gamma\in\mathbb{R}$ and is a nilpotent map $V_{\mathsf{P}_{\gamma}}(C')\rightarrow V_{\mathsf{P}_{\gamma}}(C')$. Also note that $\dim V_{\mathsf{P}_{\gamma}}(C')=N_{\Sigma}(\gamma)$ and $V_{\mathsf{P}_{\gamma}}(C')=H_{\gamma}(C')$ for $\gamma\leq 2$. The space $V_{\mathsf{P}_{\gamma}}(C')$ serves as the model space in the definition of discrete asymptotics on general Riemannian manifolds with conical singularities. 

The definition of discrete asymptotics on compact Riemannian manifolds with conical singularities is based on the following proposition.

\begin{Prop}\label{Asymptotics}
Let $(M,g)$ be a compact $m$-dimensional Riemannian manifold with conical singularities as in Definition \ref{DefinitionConicalSingularities}, $m\geq 3$, and $\boldsymbol\gamma\in\mathbb{R}^n$. Then for every $\varepsilon>0$ there exists a linear map
\begin{equation*}\label{MAP}
\Psi_{\boldsymbol\gamma}:\bigoplus_{i=1}^nV_{\mathsf{P}_{\gamma_i}}(C_i')\longrightarrow C^{\infty}(M'),
\end{equation*}
such that the following hold.\vspace{0.17cm}
\begin{compactenum}
\item[{\rm(i)}] For every $v\in\bigoplus_{i=1}^nV_{\mathsf{P}_{\gamma_i}}(C_i')$ with $v=(v_1,\ldots,v_n)$ and $v_i=r^{\beta_i}\varphi_i$ where $\varphi_i\in C^{\infty}(\Sigma_i)$ for $i=1,\ldots,n$ we have
\[
|\nabla^k(\phi_i^*(\Psi_{\boldsymbol\gamma}(v))-v_i)|=O(r^{\mu_i-\varepsilon+\beta_i-k})\quad\mbox{as }r\longrightarrow 0\mbox{ for }k\in\mathbb{N}
\]
and $i=1,\ldots,n$.
\item[{\rm(ii)}] For every $v\in\bigoplus_{i=1}^nV_{\mathsf{P}_{\gamma_i}}(C_i')$ with $v=(v_1,\ldots,v_n)$ we have 
\[
\Delta_g(\Psi_{\boldsymbol\gamma}(v))-\sum_{i=0}^n\Psi_{\boldsymbol\gamma}(\Delta_{g_i}v_i)\in C^{\infty}_{\cs}(M').
\]
\end{compactenum}
\end{Prop}

\noindent Proposition \ref{Asymptotics} is proved using the asymptotic condition \eq{ConicalSingularity} and Theorem \ref{Fredholm}. We will not give a proof here as it is mainly technical, but refer the interested reader to the author's thesis \cite[Prop. 6.14]{BehrndtDPhil}.

Using Proposition \ref{Asymptotics} we can now define weighted $C^k$-spaces, H\"older spaces, and Sobolev spaces with discrete asymptotics on compact Riemannian manifolds with conical singularities as follows. If $(M,g)$ is a compact $m$-dimensional Riemannian manifold with conical singularities, $m\geq 3$, then for $k\in\mathbb{N}$, $\alpha\in(0,1)$, and $\boldsymbol\gamma\in\mathbb{R}^n$ we define
\[
C^k_{\boldsymbol\gamma,\mathsf{P}_{\boldsymbol\gamma}}(M')=C^k_{\boldsymbol\gamma}(M')\oplus\im\;\Psi_{\boldsymbol\gamma}\quad\mbox{and}\quad C^{k,\alpha}_{\boldsymbol\gamma,\mathsf{P}_{\boldsymbol\gamma}}(M')=C^{k,\alpha}_{\boldsymbol\gamma}(M')\oplus\im\;\Psi_{\boldsymbol\gamma}.
\] 
Finally if $p\in[1,\infty)$, then we define the weighted Sobolev space with discrete asymptotics $W^{k,p}_{\boldsymbol\gamma,\mathsf{P}_{\boldsymbol\gamma}}(M')$ by
\[
W^{k,p}_{\boldsymbol\gamma,\mathsf{P}_{\boldsymbol\gamma}}(M')=W^{k,p}_{\boldsymbol\gamma}(M')\oplus\im\;\Psi_{\boldsymbol\gamma}.
\]
Then $C^k_{\boldsymbol\gamma,\mathsf{P}_{\boldsymbol\gamma}}(M')$, $C^{k,\alpha}_{\boldsymbol\gamma,\mathsf{P}_{\boldsymbol\gamma}}(M')$, and $W^{k,p}_{\boldsymbol\gamma,\mathsf{P}_{\boldsymbol\gamma}}(M')$ are Banach spaces, where the norm on the discrete asymptotics part is some finite dimensional norm. Note that the discrete asymptotics are trivial if $\boldsymbol\gamma\leq 0$, so that in this case the weighted spaces with discrete asymptotics are simply weighted spaces as defined in \S\ref{WeightedSpaces}.

\subsection{The Laplace operator on compact Riemannian manifolds with conical singularities. II}

If $(M,g)$ is a compact Riemannian manifold, then it is well known that the Laplace operator defines an isomorphism of Banach spaces
\[
\Delta_g:\left\{u\in W^{k,p}(M)\;:\;\mbox{$\int_Mu\;\mathrm dV_g=0$}\right\}\longrightarrow\left\{u\in W^{k-2,p}(M)\;:\;\mbox{$\int_Mu\;\mathrm dV_g=0$}\right\}
\]
for every $k\in\mathbb{N}$ with $k\geq 2$ and $p\in(1,\infty)$. The result continues to hold if we replace the Sobolev spaces by H\"older spaces, see Aubin \cite[Ch. 4, Thm. 4.7]{Aubin}. Using the weighted H\"older and Sobolev spaces with discrete asymptotics we can now state a similar result for the Laplace operator on Riemannian manifolds with conical singularities.

Let $(M,g)$ be a compact $m$-dimensional Riemannian manifold with conical singularities as in Definition \ref{DefinitionConicalSingularities}, $m\geq 3$, and $\boldsymbol\gamma\in\mathbb{R}^n$ with $\boldsymbol\gamma>2-m$. For $k\in\mathbb{N}$ and $\alpha\in(0,1)$ we then define $C^{k,\alpha}_{\boldsymbol\gamma,\mathsf{P}_{\boldsymbol\gamma}}(M')_0$ to consist of $u\in C^{k,\alpha}_{\boldsymbol\gamma,\mathsf{P}_{\boldsymbol\gamma}}(M')$ with $\int_{M'}u\;\mathrm dV_g=0$, and in a similar way for $p\in[1,\infty)$ we let $W^{k,p}_{\boldsymbol\gamma,\mathsf{P}_{\boldsymbol\gamma}}(M')_0$ be the space of functions $u\in W^{k,p}_{\boldsymbol\gamma,\mathsf{P}_{\boldsymbol\gamma}}(M')$ with $\int_{M'}u\;\mathrm dV_g=0$. From Proposition \ref{Asymptotics} it follows that for $k\in\mathbb{N}$ with $k\geq 2$ and $\alpha\in(0,1)$, $\Delta_g:C^{k,\alpha}_{\boldsymbol\gamma,\mathsf{P}_{\boldsymbol\gamma}}(M')_0\rightarrow C^{k-2,\alpha}_{\boldsymbol\gamma-2,\mathsf{P}_{\boldsymbol\gamma-\boldsymbol 2}}(M')_0$ and for $p\in[1,\infty)$, $\Delta_g:W^{k,p}_{\boldsymbol\gamma,\mathsf{P}_{\boldsymbol\gamma}}(M')_0\rightarrow W^{k-2,p}_{\boldsymbol\gamma-2,\mathsf{P}_{\boldsymbol\gamma-\boldsymbol 2}}(M')_0$ are well defined linear operators. We then have the following result, which also verifies \eq{llkko}.

\begin{Prop}\label{Invertibility}
Let $(M,g)$ be a compact $m$-dimensional Riemannian manifold with conical singularities, $m\geq 3$, and $\boldsymbol\gamma\in\mathbb{R}^n$ with $\boldsymbol\gamma>2-m$ and $\gamma_i\notin\mathcal{E}_{\Sigma_i}$ for $i=1,\ldots,n$. Then the following hold.\vspace{0.17cm}
\begin{compactenum}
\item[{\rm(i)}] Let $k\in\mathbb{N}$ with $k\geq 2$ and $\alpha\in(0,1)$. Then 
\begin{equation}\label{LLOPI1}
\Delta_g:C^{k,\alpha}_{\boldsymbol\gamma,\mathsf{P}_{\boldsymbol\gamma}}(M')_0\rightarrow C^{k-2,\alpha}_{\boldsymbol\gamma-2,\mathsf{P}_{\boldsymbol\gamma-2}}(M')_0
\end{equation}
is an isomorphism of Banach spaces.
\item[{\rm(ii)}] Let $k\in\mathbb{N}$ with $k\geq 2$ and $p\in(1,\infty)$. Then
\begin{equation}\label{LLOPI2}
\Delta_g:W^{k,p}_{\boldsymbol\gamma,\mathsf{P}_{\boldsymbol\gamma}}(M')_0\rightarrow W^{k-2,p}_{\boldsymbol\gamma-2,\mathsf{P}_{\boldsymbol\gamma-2}}(M')_0
\end{equation}
is an isomorphism of Banach spaces.
\end{compactenum} 
\end{Prop}

\begin{proof}
We demonstrate the proof of {\rm(i)}, the proof of {\rm(ii)} goes similarly. Thus let $k\in\mathbb{N}$ with $k\geq 2$, $\alpha\in(0,1)$, and $\boldsymbol\gamma\in\mathbb{R}^n$ with $\boldsymbol\gamma>2-m$ and $\gamma_i\notin\mathcal{E}_{\Sigma_i}$ for $i=1,\ldots,n$. Then by Theorem \ref{Fredholm}, $\Delta_g:C^{k,\alpha}_{\boldsymbol\gamma}(M')_0\longrightarrow C^{k-2,\alpha}_{\boldsymbol\gamma-2}(M')_0$ is a Fredholm operator and
\begin{equation}\label{Index1}
\index\left\{\Delta_g:C^{k,\alpha}_{\boldsymbol\gamma}(M')_0\longrightarrow C^{k-2,\alpha}_{\boldsymbol\gamma-2}(M')_0\right\}=-\sum_{i=1}^nM_{\Sigma_i}(\gamma_i).
\end{equation}
When we replace $C^{k,\alpha}_{\boldsymbol\gamma-2}(M')_0$ by $C^{k,\alpha}_{\boldsymbol\gamma-2,\mathsf{P}_{\boldsymbol\gamma-2}}(M')_0$ and $C^{k,\alpha}_{\boldsymbol\gamma}(M')_0$ by the space $C^{k,\alpha}_{\boldsymbol\gamma,\mathsf{P}_{\boldsymbol\gamma}}(M')_0$ in \eq{Index1}, then
\begin{gather}\label{Index3}
\begin{split}
&\index\left\{\Delta_g:C^{k,\alpha}_{\boldsymbol\gamma,\mathsf{P}_{\boldsymbol\gamma}}(M')_0\rightarrow C^{k-2,\alpha}_{\boldsymbol\gamma-2,\mathsf{P}_{\boldsymbol\gamma-2}}(M')_0\right\}=\quad\quad\quad\quad\quad\\&\quad\quad\quad\quad\quad\quad\quad\quad-\sum_{i=1}^nM_{\Sigma_i}(\gamma_i)-\dim\im\;\Psi_{\boldsymbol\gamma-2}+\dim\im\;\Psi_{\boldsymbol\gamma}
\end{split}
\end{gather}
By definition of $\Psi_{\boldsymbol\gamma}$ and $\Psi_{\boldsymbol\gamma-2}$ we have that $\dim\im\;\Psi_{\boldsymbol\gamma}=\sum_{i=1}^nN_{\Sigma_i}(\gamma_i)$ and $\dim\im\;\Psi_{\boldsymbol\gamma-2}=\sum_{i=1}^nN_{\Sigma_i}(\gamma_i-2)$, where $N_{\Sigma_i}$ is defined in \eq{DEFN}. Using \eq{indexxx} we then conclude from \eq{Index3} that 
\[
\index\left\{\Delta_g:C^{k,\alpha}_{\boldsymbol\gamma,\mathsf{P}_{\boldsymbol\gamma}}(M')_0\rightarrow C^{k-2,\alpha}_{\boldsymbol\gamma-2,\mathsf{P}_{\boldsymbol\gamma-2}}(M')_0\right\}=0.
\]
Thus \eq{llkko} holds and in order to show that $\Delta_g:C^{k,\alpha}_{\boldsymbol\gamma,\mathsf{P}_{\boldsymbol\gamma}}(M')_0\rightarrow C^{k-2,\alpha}_{\boldsymbol\gamma-2,\mathsf{P}_{\boldsymbol\gamma-2}}(M')_0$ is a bijection it suffices to show that the kernel is trivial. 

Let $u\in C^{k,\alpha}_{\boldsymbol\gamma,\mathsf{P}_{\boldsymbol\gamma}}(M')_0$, such that $\Delta_gu=0$, and let us first assume that $\boldsymbol\gamma>\frac{1}{2}(2-m)$. Then integration by parts gives
\[
0=\int_{M'}u\Delta_gu\;\mathrm dV_g=-\int_{M'}|\mathrm du|^2\;\mathrm dV_g
\]
and hence $\mathrm du=0$. So $u$ is constant on $M'$, but $\int_{M'}u\;\mathrm dV_g=0$, and hence $u\equiv 0$. Since $(2-m,0)^n$ is a connected subset of $(\mathbb{R}^n\backslash\mathcal{D}_{\Sigma_1})\times\cdots\times(\mathbb{R}^n\backslash\mathcal{D}_{\Sigma_n})$ that contains $(\frac{1}{2}(2-m),\ldots,\frac{1}{2}(2-m))$, it follows from Theorem \ref{Fredholm} that the kernel of $\Delta_g:C^{k,\alpha}_{\boldsymbol\gamma,\mathsf{P}_{\boldsymbol\gamma}}(M')_0\rightarrow C^{k-2,\alpha}_{\boldsymbol\gamma-2,\mathsf{P}_{\boldsymbol\gamma-2}}(M')_0$ is trivial for every $\boldsymbol\gamma>2-m$. Hence $\Delta_g:C^{k,\alpha}_{\boldsymbol\gamma,\mathsf{P}_{\boldsymbol\gamma}}(M')_0\rightarrow C^{k-2,\alpha}_{\boldsymbol\gamma-2,\mathsf{P}_{\boldsymbol\gamma-2}}(M')_0$ is a bijection, and the Open Mapping Theorem \cite[Ch. II, \S 5]{Yosida} implies that this operator is an isomorphism of Banach spaces.
\end{proof}

The next proposition is a version of the Schauder and $L^p$-estimates for the Laplace operator acting on weighted spaces with discrete asymptotics.

\begin{Prop}\label{SchauderLPLaplace2}
Let $(M,g)$ be a compact $m$-dimensional Riemannian manifold with conical singularities as in Definition \ref{DefinitionConicalSingularities}, $m\geq 3$, and $\boldsymbol\gamma\in\mathbb{R}^n$ with $\gamma_i\notin\mathcal{E}_{\Sigma_i}$ for $i=1,\ldots,n$. Let $u,f\in L^1_{\loc}(M')$ and assume that $\Delta_gu=f$ holds in the weak sense. Then the following hold.\vspace{0.17cm}
\begin{compactenum}
\item[{\rm(i)}] Let $k\in\mathbb{N}$ with $k\geq 2$ and $\alpha\in(0,1)$. If $f\in C^{k-2,\alpha}_{\boldsymbol\gamma-2,\mathsf{P}_{\boldsymbol\gamma-2}}(M')$ and $u\in C^0_{\boldsymbol\gamma,\mathsf{P}_{\boldsymbol\gamma}}(M')$, then $u\in C^{k,\alpha}_{\boldsymbol\gamma,\mathsf{P}_{\boldsymbol\gamma}}(M')$. Moreover there exists a constant $c>0$ independent of $u$ and $f$, such that
\begin{equation}\label{estimate3}
\|u\|_{C^{k,\alpha}_{\boldsymbol\gamma,\mathsf{P}_{\boldsymbol\gamma}}}\leq c\left(\|f\|_{C^{k-2,\alpha}_{\boldsymbol\gamma-2,\mathsf{P}_{\boldsymbol\gamma-2}}}+\|u\|_{C^0_{\boldsymbol\gamma,\mathsf{P}_{\boldsymbol\gamma}}}\right).
\end{equation}
\item[{\rm(ii)}] Let $k\in\mathbb{N}$ with $k\geq 2$ and $p\in(1,\infty)$. If $f\in W^{k-2,p}_{\boldsymbol\gamma-2,\mathsf{P}_{\boldsymbol\gamma-2}}(M')$ and $u\in L^p_{\boldsymbol\gamma,\mathsf{P}_{\boldsymbol\gamma}}(M')$, then $u\in W^{k,p}_{\boldsymbol\gamma,\mathsf{P}_{\boldsymbol\gamma}}(M')$. Moreover there exists a constant $c>0$ independent of $u$ and $f$, such that
\begin{equation}\label{estimate4}
\|u\|_{W^{k,p}_{\boldsymbol\gamma,\mathsf{P}_{\boldsymbol\gamma}}}\leq c\left(\|f\|_{W^{k-2,p}_{\boldsymbol\gamma-2,\mathsf{P}_{\boldsymbol\gamma-2}}}+\|u\|_{L^p_{\boldsymbol\gamma,\mathsf{P}_{\boldsymbol\gamma}}}\right).
\end{equation}
\end{compactenum}
\end{Prop}

\begin{proof}
We demonstrate the proof of \rm{(i)}, the proof of \rm{(ii)} goes similarly. We can assume that $\boldsymbol\gamma>0$, since otherwise the discrete asymptotics are trivial and we are in the situation of Proposition \ref{SchauderLPLaplace}. Let $k\in\mathbb{N}$ with $k\geq 2$, $\alpha\in(0,1)$, and assume that $f\in C^{k-2,\alpha}_{\boldsymbol\gamma-2,\mathsf{P}_{\boldsymbol\gamma-2}}(M')$ and $u\in C^0_{\boldsymbol\gamma,\mathsf{P}_{\boldsymbol\gamma}}(M')$. Using that the discrete asymptotics are bounded functions on $M'$ and the weighted Schauder estimates from \eq{estimate} we find that $u\in C^{k,\alpha}_{\boldsymbol 0}(M')$. Hence $\Delta_gu=f$ and $\int_{M'}f\;\mathrm dV_g=0$. Choose $\phi\in C^{\infty}_{\cs}(M')$ with $\int_{M'}\phi\;\mathrm dV_g=1$ and write $u=u_0+\lambda\phi$ with $u_0\in C^{k,\alpha}_{\boldsymbol 0}(M')_0$ and $\lambda\in\mathbb{R}$. Then Proposition \ref{Invertibility}, \rm{(i)}, implies $u_0\in C^{k,\alpha}_{\boldsymbol\gamma,\mathsf{P}_{\boldsymbol\gamma}}(M')_0$ and thus $u\in C^{k,\alpha}_{\boldsymbol\gamma,\mathsf{P}_{\boldsymbol\gamma}}(M')$ as we wanted to show. 

It remains to prove the estimate \eq{estimate3}. Write $u=u_1+u_2$ with $u_1\in C^{k,\alpha}_{\boldsymbol\gamma}(M')$, $u_2\in\im\;\Psi_{\boldsymbol\gamma}$ and $f=f_1+f_2$ with $f_1\in C^{k-2,\alpha}_{\boldsymbol\gamma-2}(M')$ and $f_2\in\im\;\Psi_{\boldsymbol\gamma-2}$. Then 
\[
\Delta_gu_1+\pi_{C^{k-2,\alpha}_{\boldsymbol\gamma-2}}(\Delta_gu_2)=f_1\quad\mbox{and}\quad\pi_{\im\;\Psi_{\boldsymbol\gamma-2}}(\Delta_gu_2)=f_2,
\]
where $\pi_X$ denotes the projection onto the space $X$. Using the weighted Schauder estimates and the continuity of the linear operator $\pi_{C^{k-2,\alpha}_{\boldsymbol\gamma-2}}\circ\Delta_g:\im\;\Psi_{\boldsymbol\gamma}\rightarrow C^{k-2,\alpha}_{\boldsymbol\gamma-2}(M')$ we find
\begin{align*}
\|u_1\|_{C^{k,\alpha}_{\boldsymbol\gamma}}&\leq c\left(\|f_1\|_{C^{k-2,\alpha}_{\boldsymbol\gamma-2}}+\|\pi_{C^{k-2,\alpha}_{\boldsymbol\gamma-2}}(\Delta_gu_2)\|_{C^{k-2,\alpha}_{\boldsymbol\gamma-2}}+\|u_1\|_{C^0_{\boldsymbol\gamma}}\right)\\&\leq c\left(\|f_1\|_{C^{k-2,\alpha}_{\boldsymbol\gamma-2}}+\|u_2\|_{\im\;\Psi_{\boldsymbol\gamma}}+\|u_1\|_{C^0_{\boldsymbol\gamma}}\right)\\&=c\left(\|f_1\|_{C^{k-2,\alpha}_{\boldsymbol\gamma-2}}+\|u\|_{C^0_{\boldsymbol\gamma,\mathsf{P}_{\boldsymbol\gamma}}}\right).
\end{align*}
Lastly we estimate $u_2$ in terms of $f$. Choose some small $\varepsilon>0$ such that $[\gamma_i-\varepsilon,\gamma_i]\cap\mathcal{D}_{\Sigma_i}=\emptyset$ for $i=1,\ldots,n$. Then by Theorem \ref{Fredholm} and Proposition \ref{Cokernel}, $\Delta_g:W^{k,2}_{\boldsymbol\gamma-\varepsilon}(M')\rightarrow W^{k-2,2}_{\boldsymbol\gamma-\varepsilon-2}(M')$ is a Fredholm operator with cokernel being isomorphic to the kernel of $\Delta_g:W^{k,2}_{2-m-\boldsymbol\gamma+\varepsilon}(M)\rightarrow W^{k-2,2}_{-m-\boldsymbol\gamma+\varepsilon}(M')$. Since $u_1\in C^{k,\alpha}_{\boldsymbol\gamma}(M')$, also $u_1\in W^{k,p}_{\boldsymbol\gamma-\varepsilon}(M')$. Using integration by parts we therefore find that 
\[
\langle h,f_1\rangle_{L^2}=\langle h,\Delta_gu_1+\pi_{C^{k-2,p}_{\boldsymbol\gamma-2}}(\Delta_gu_2)\rangle_{L^2}=\langle h,\pi_{C^{k-2,p}_{\boldsymbol\gamma-2}}(\Delta_gu_2)\rangle_{L^2}
\]
for $h\in\ker\{\Delta_g:W^{k,2}_{2-m-\boldsymbol\gamma+\varepsilon}(M')\rightarrow W^{k-2,2}_{-m-\boldsymbol\gamma+\varepsilon}(M')\}$.
Therefore $f_1$ determines $\sum_{i=1}^nM_{\Sigma_i}(\gamma_i)$ components of $u_2$. Moreover $\pi_{\im\;\Psi_{\boldsymbol\gamma-2}}(\Delta_gu_2)=f_2$, and hence $f_2$ determines $\sum_{i=1}^nN_{\Sigma_i}(\gamma_i-2)$ different components of $u_2$. Thus by \eq{indexxx}, $f$ determines $\sum_{i=1}^nN_{\Sigma_i}(\gamma_i)$ components of $u_2$. Since $\dim\im\;\Psi_{\boldsymbol\gamma}=\sum_{i=1}^nN_{\Sigma_i}(\gamma_i)$, $u_2$ is uniquely determined by $f$. Hence $\|u_2\|_{\im\;\Psi_{\boldsymbol\gamma}}\leq c\|f\|_{C^{k-2,\alpha}_{\boldsymbol\gamma-2,\mathsf{P}_{\boldsymbol\gamma-2}}}$.
\end{proof}

\section{The Friedrichs heat kernel on Riemannian manifolds with conical singularities}\label{HEATKERNEL}
\subsection{The Friedrichs heat kernel on Riemannian manifolds}
Let $(M,g)$ be a Riemannian manifold and consider the Laplace operator acting as an unbounded operator $
\Delta_g:C^{\infty}_{\cs}(M)\subset L^2(M)\rightarrow L^2(M)$, where $C^{\infty}_{\cs}(M)$ is the space of smooth functions on $M$ with compact support. This is a symmetric and nonpositive operator and by Friedrichs' theorem \cite[Ch. XI, \S 7, Thm. 2]{Yosida} there exists a closed and self-adjoint extension $\Delta_g:\dom(\Delta_g)\subset L^2(M)\rightarrow L^2(M)$, called the Friedrichs extension. Then by the spectral theorem for self-adjoint operators \cite[Ch. XI, \S 6, Thm. 1]{Yosida} there exists a unique resolution of the identity $\{E_{\lambda}\}_{\lambda\in\mathbb{R}}$ such that $\Delta_g=\int_{-\infty}^{\infty}\lambda\;\mathrm dE_{\lambda}$. Using the functional calculus for self-adjoint operators \cite[Ch XI, \S 12]{Yosida} we define the Friedrichs heat semigroup $\{\exp(t\Delta_g)\}_{t>0}$ by $\exp(t\Delta_g)=\int_{-\infty}^{\infty}\exp(t\lambda)\;\mathrm dE_\lambda$. Then $\{\exp(t\Delta_g)\}_{t>0}$ is a semigroup of bounded operators on $L^2(M)$ that maps $L^2(M)$ into $\bigcap_{j=0}^\infty\dom(\Delta_g^j)$ for every $t>0$.

The next proposition shows that the action of the Friedrichs heat semigroup on $L^2(M)$ is given by an integral operator with a positive and symmetric kernel.

\begin{Prop}\label{ExistenceHeatKernel}
Let $(M,g)$ be a Riemannian manifold and $\{\exp(t\Delta_g)\}_{t>0}$ the Friedrichs heat semigroup on $(M,g)$. Then there exists a positive function $H\in C^{\infty}((0,\infty)\times M\times M)$, which is symmetric on $M\times M$, such that for every $\varphi\in L^2(M)$
\begin{equation}\label{HeatSemigroup4}
(\exp(t\Delta_g)\varphi)(x)=\int_MH(t,x,y)\varphi(y)\;\mathrm dV_g(y).
\end{equation}
The function $H$ is called the Friedrichs heat kernel on $(M,g)$.
\end{Prop}

\noindent The proof of Proposition \ref{ExistenceHeatKernel} can be found in Davies \cite[Thm 5.2.1]{Davies}.

There is a well known parametrix construction for the heat kernel, which can be found in Aubin \cite[Ch 4, \S 4.2]{Aubin}.

\begin{Thm}\label{Parametrix}
Let $(M,g)$ be an $m$-dimensional Riemannian manifold and let $H$ be the Friedrichs heat kernel on $(M,g)$. Then near the diagonal in $M\times M$, $H$ has an asymptotic expansion as $t\rightarrow 0$ of the form
\begin{equation}\label{asymptotic}
H(t,x,y)\sim\frac{1}{(4\pi t)^{m/2}}\exp\left(-\frac{d_g(x,y)^2}{4t}\right)\sum_{j=0}^{\infty}a_j(x,y)t^j,
\end{equation}
where $a_j\in C^{\infty}(M\times M)$ for $j\in\mathbb{N}$ and $a_0(x,x)=1$ for $x\in M$.
\end{Thm}

Using Theorem \ref{Parametrix} and the standard regularity theory for the heat equation on domains in $\mathbb{R}^m$, see for instance Krylov \cite[Ch. 5]{KrylovSobolev} and \cite[Ch. 8]{KrylovHoelder}, it is straightforward to prove existence and maximal regularity of solutions for the Cauchy problem \eq{HeatEquation} on compact Riemannian manifolds, when the free term lies in a parabolic H\"older or Sobolev space.

\subsection{Asymptotics of the Friedrichs heat kernel on compact Riemannian manifolds with conical singularities}
In this subsection we discuss in an informal way the parametrix construction for the Friedrichs heat kernel on compact Riemannian manifolds with conical singularities following Mooers \cite{Mooers}.

Let $(M,g)$ be a compact $m$-dimensional Riemannian manifold with conical singularities as in Definition \ref{DefinitionConicalSingularities}, $m\geq 3$, and let $H$ be the Friedrichs heat kernel on $(M,g)$, which is a smooth function on $(0,\infty)\times M'\times M'$. Our goal is to understand the asymptotics of $H(t,x,y)$ for $x,y\in M$ as $t\rightarrow 0$. First assume that $x$ and $y$ are close to each other and away from the singularities. Then Theorem \ref{Parametrix} gives a description of the asymptotics of $H(t,x,y)$ as $t\rightarrow 0$. Loosely speaking, if $x$ and $y$ are away from the conical singularities and close to each other, then around $x$ and $y$, $M$ is modelled on the Euclidean space, and therefore the Euclidean heat kernel gives a parametrix for the heat kernel. 

Next let us consider the case, when $x$ and $y$ are close to a singularity. In this case the right model to consider is not the Euclidean space anymore, but the Riemannian cone that models the conical singularity, and we are therefore led to study the Friedrichs heat kernel on Riemannian cones. Let $(\Sigma,h)$ be a compact and connected $(m-1)$-dimensional Riemannian manifold, $m\geq 3$, $(C,g_C)$ the Riemannian cone over $(\Sigma,h)$, and $H_C$ the Friedrichs heat kernel on $(C,g_C)$. A main feature of the Laplace operator on Riemannian cones is its dilation equivariance in radial directions. To exhibit how the dilation equivariance effects the Friedrichs heat kernel let us define an action $\delta^s$ of $s\in(0,\infty)$ on $(0,\infty)\times C'\times C'$ by $\delta^s(t,\sigma,r,\sigma',r')=(s^2t,\sigma,sr,\sigma',sr')$. 
If $t>0$, then 
\begin{equation}\label{IGOR}
(\delta^s)^*(t\Delta_{g_C}\varphi)(\sigma,r)=t\Delta_{g_C}(\delta^s)^*(\varphi)(\sigma,r)\quad\mbox{for }(\sigma,r)\in C'
\end{equation}
and $\varphi\in\dom(\Delta_{g_C})$. Here $\dom(\Delta_{g_C})$ is the domain of the Friedrichs extension of the Laplace operator $\Delta_{g_C}:C^{\infty}_{\cs}(C')\subset L^2(C')\rightarrow L^2(C')$. Then \eq{IGOR} implies that 
\begin{equation}\label{POOHKISSEN}
(\delta^s)^*(\exp(t\Delta_{g_C})\varphi)=\exp(t\Delta_{g_C})(\delta^s)^*(\varphi)\quad\mbox{for }\varphi\in\dom(\Delta_{g_C}).
\end{equation}
From \eq{POOHKISSEN} and Proposition \ref{ExistenceHeatKernel} we then conclude that
\begin{equation}\label{HERMANNSCHWEIN}
(\delta^s)^*(H_C)(t,\sigma,r,\sigma',r')=H_C(s^2,sr,\sigma,sr',\sigma')=s^{-m}H_C(t,r,\sigma,r',\sigma')
\end{equation}
for $(t,\sigma,r,\sigma',r')\in(0,\infty)\times C'\times C'$. Thus, at least in an asymptotic sense, the Friedrichs heat kernel $H$ on $(M,g)$ should also satisfy the homogeneity relation \eq{HERMANNSCHWEIN}. Such a result can be deduced from Nagase \cite[\S 5]{Nagase}.

Finally let us study the function $x\mapsto H(t,x,y)$ for fixed $t>0$ and $y\in M'$. Recall that the Friedrichs heat semigroup $\{\exp(t\Delta_g)\}_{t>0}$ is a semigroup of bounded operators on $L^2(M')$. Moreover for every $t>0$, $\exp(t\Delta_g)$ maps $L^2(M')$ into $\bigcap_{j=0}^\infty\dom(\Delta_g^j)$, and Proposition \ref{ExistenceHeatKernel} implies that for fixed $t>0$ and $y\in M'$ the function $x\mapsto H(t,x,y)$ lies in $\bigcap_{j=0}^\infty\dom(\Delta_g^j)$. By Proposition \ref{Invertibility} and Theorem \ref{WeightedSobolevEmbedding} we have that $\bigcap_{j=0}^\infty\dom(\Delta_g^j)=\bigcap_{\boldsymbol\gamma\in\mathbb{R}^n}C^{\infty}_{\boldsymbol\gamma,\mathsf{P}_{\boldsymbol\gamma}}(M')$, and hence the function $x\mapsto H(t,x,y)$ lies in $\bigcap_{\boldsymbol\gamma\in\mathbb{R}^n}C^{\infty}_{\boldsymbol\gamma,\mathsf{P}_{\boldsymbol\gamma}}(M')$ for fixed $t>0$ and $y\in M'$.

We now discuss parts of Mooers' parametrix construction for the Friedrichs heat kernel \cite{Mooers}. We explain this construction only in an informal way and the interested reader should consult Mooers paper for a detailed description. In order to describe the asymptotics of the Friedrichs heat kernel it is convenient to introduce functions $\rho_{\bff},\rho_{\tf},\rho_{\lb},\rho_{\rb}$, and $\rho_{\tb}$ on $(0,\infty)\times M'\times M'$ as follows. Let $\rho$ be a radius function on $M'$ and define
\begin{eqnarray*}
&&\rho_{\bff}(t,x,y)=\sqrt{t+\rho(x)^2+\rho(y)^2},\quad\rho_{\tf}(t,x,y)=\frac{\sqrt{t+d_g(x,y)^2}}{\sqrt{t+\rho(x)^2+\rho(y)^2}},\\
&&\rho_{\lb}(t,x,y)=\frac{\rho(x)}{\sqrt{t+\rho(x)^2+\rho(y)^2}},\quad\rho_{\rb}(t,x,y)=\frac{\rho(y)}{\sqrt{t+\rho(x)^2+\rho(y)^2}},\\
&&\rho_{\tb}(t,x,y)=\frac{\sqrt{t}}{\sqrt{t+d_g(x,y)^2}}.
\end{eqnarray*}
Loosely speaking we have that $\rho_{\bff}(t,x,y)=0$ if and only if $t=0$ and $\rho(x)=\rho(y)=0$, $\rho_{\tf}(t,x,y)=0$ if and only if $t=0$ and $x=y$, $\rho_{\lb}(t,x,y)=0$ if and only if $\rho(x)=0$, $\rho_{\rb}(t,x,y)=0$ if and only if $\rho(y)=0$, and finally $\rho_{\tb}(t,x,y)=0$ if and only if $t=0$ and $x\neq y$. In fact the functions $\rho_{\bff},\rho_{\tf},\rho_{\lb},\rho_{\rb}$, and $\rho_{\tb}$ should be understood as boundary defining functions on the heat space of $M$, see Melrose \cite[Ch. 7, \S 4]{Melrose1}.

From Theorem \ref{Parametrix} we have a good understanding of the asymptotics of $H(t,x,y)$, when $x$ and $y$ lie in a compact region, so we only have to study the asymptotics of the heat kernel, when $x$ or $y$ are close to a singularity. The first step in the parametrix construction for the heat kernel is to find a rough parametrix $H_0$, i.e. a good first approximation, for $H$. The rough parametrix $H_0$ is constructed by gluing the heat kernels on the model cones of the conical singularities together with the heat kernel $H$. Since the Laplace operator on $M'$ near each conical singularity is asymptotic to the Laplace operator on the model cone of the singularity, it follows that $H_0$ is a good first approximation for the heat kernel $H$ and determines the leading order terms in the asymptotic expansion of $H$ in terms of $\rho_{\bff},\rho_{\tf},\rho_{\lb},\rho_{\rb}$, and $\rho_{\tb}$. Using the discussion from above, we have a good understanding of the asymptotics of $H_0$, and, in fact, one can determine the expansion of $H_0$ in terms of the functions $\rho_{\bff},\rho_{\tf},\rho_{\lb},\rho_{\rb}$, and $\rho_{\tb}$ and show that $H_0\sim\rho_{\tf}^{-m}\rho_{\bff}^{-m}\rho_{\tb}^{\infty}\rho_{\lb}^0\rho_{\rb}^0$, see Mooers \cite[Prop. 3.3]{Mooers}. (Note, however, that due to a mistake in \cite[Lem. 3.2]{Mooers} the power $-1$ of the function $\rho_{\bff}$ in Mooers' result should be replaced by $-m$). What is left, is to solve away the error terms caused by the gluing procedure and the asymptoticness of the Laplace operator on $M'$ to the Laplace operators on the model cones. This is done in Mooers \cite[Prop. 3.4 -- 3.8]{Mooers}. 

Of particular importance for us are the asymptotics of $H$ when $\rho_{\lb},\rho_{\rb}\rightarrow 0$, since this is where the discrete asymptotics come into play. Let $\boldsymbol\gamma\in\mathbb{R}^n$ and define $\boldsymbol{\gamma^+},\boldsymbol{\gamma^-}\in\mathbb{R}^n$ by
\begin{equation}\label{BIPPELKISSEN}
\gamma_i^+=\min\left\{\varepsilon\in\mathcal{E}_{\Sigma_i}\;:\;\varepsilon\geq\gamma_i\right\}\quad\mbox{and}\quad\gamma_i^-=\max\left\{\varepsilon\in\mathcal{E}_{\Sigma_i}\;:\;\varepsilon<\gamma_i\right\}
\end{equation}
for $i=1,\ldots,n$. For $\boldsymbol\gamma\in\mathbb{R}^n$ we choose a basis $\psi_{\boldsymbol\gamma}^1,\ldots,\psi_{\boldsymbol\gamma}^N$ for $\im\;\Psi_{\boldsymbol\gamma}$, where $N=\dim\im\;\Psi_{\boldsymbol\gamma}$. Recall from above that the function $x\mapsto H(t,x,y)$ lies in $\bigcap_{\boldsymbol\gamma\in\mathbb{R}^n}C^{\infty}_{\boldsymbol\gamma,\mathsf{P}_{\boldsymbol\gamma}}(M')$ for fixed $t>0$ and $y\in M'$. Now one can deduce from \cite[Prop. 3.5]{Mooers} that there exist functions $H_{\boldsymbol\gamma}^1,\ldots,H_{\boldsymbol\gamma}^N\in C^{\infty}((0,\infty)\times M')$ that admit an asymptotic expansion of the form
\begin{equation}\label{HERMANN}
H^j_{\boldsymbol\gamma}\sim\rho_{\tf}^{-m}\rho_{\bff}^{-m}\rho_{\tb}^{\infty}\rho_{\rb}^{-\boldsymbol\gamma^-}\quad\mbox{for }j=1,\ldots,N,
\end{equation}
and such that we have an asymptotic expansion of the form
\begin{equation}\label{HERMANNSCHWEIN2}
H-\sum\nolimits_{j=1}^N\psi_{\boldsymbol\gamma}^jH^j_{\boldsymbol\gamma}\sim\rho_{\tf}^{-m}\rho_{\bff}^{-m}\rho_{\tb}^{\infty}\rho_{\lb}^{\boldsymbol\gamma^+}.
\end{equation}
The time derivatives of $H$ then admit a similar expansion, and from \eq{HERMANN} and \eq{HERMANNSCHWEIN2} we then deduce the following result. 

\begin{Thm}\label{HeatKernelEstimates}
Let $(M,g)$ be a compact $m$-dimensional Riemannian manifold with conical singularities as in Definition \ref{DefinitionConicalSingularities}, $m\geq 3$, $H$ the Friedrichs heat kernel on $(M,g)$, and $\boldsymbol\gamma\in\mathbb{R}^n$. For $l\in\mathbb{N}$ choose a basis $\psi_{\boldsymbol\gamma-2l}^1,\ldots,\psi_{\boldsymbol\gamma-2l}^{N_l}$ for $\im\;\Psi_{\boldsymbol\gamma-2l}$, where $N_l=\dim\im\;\Psi_{\boldsymbol\gamma-2l}$. Then the following holds.

For each $l\in\mathbb{N}$ there exist functions $H_{\boldsymbol\gamma-2l}^1,\ldots,H_{\boldsymbol\gamma-2l}^{N_l}\in C^{\infty}((0,\infty)\times M')$ and constants $c_l>0$, such that for each $l\in\mathbb{N}$
\[
|H^j_{\boldsymbol\gamma-2l}(t,y)|\leq c_l\cdot(t+\rho(y)^2)^{-\frac{m+(\boldsymbol\gamma-2l)^-}{2}}\quad\mbox{for }t>0,\;y\in M'
\]
and $j=1,\ldots,N_l$, and 
\begin{align*}
&\Bigl|\partial_t^lH(t,x,y)-\sum\nolimits_{j=1}^{N_l}\psi^j_{\boldsymbol\gamma-2l}(x)H^j_{\boldsymbol\gamma-2l}(t,y))\Bigr|\\&\quad\quad\quad\quad\quad\quad\quad\leq c_l\cdot(t+d_g(x,y)^2)^{-\frac{m+l}{2}}\left(\frac{\rho(x)^2}{\rho(x)^2+\rho(y)^2}\right)^{\frac{(\boldsymbol\gamma-2l)^+}{2}},
\end{align*}
for $t>0$, and $x,y\in M'$. Here $\boldsymbol\gamma^+$ and $\boldsymbol\gamma^-$ are given in \eq{BIPPELKISSEN}.
\end{Thm}

\section{The Cauchy problem for the heat equation}\label{CAUCHY}
\subsection{Weighted parabolic H\"older and Sobolev spaces with discrete asymptotics}
We first define $C^k$-spaces, H\"older spaces, and Sobolev spaces of maps $u:I\rightarrow X$, where $I\subset\mathbb{R}$ is an open and bounded interval and $X$ is a Banach space. For $k\in\mathbb{N}$ we define $C^k_{\loc}(I;X)$ to be the space of $k$-times continuously differentiable maps $u:I\rightarrow X$. We define the $C^k$-norm by
\[
\|u\|_{C^k}=\sum_{j=0}^k\sup_{t\in I}\|\partial_t^ju(t)\|_X\quad\mbox{for }u\in C^k_{\loc}(I;X),
\]
whenever it is finite, and we define
\[
C^{k}(I;X)=\left\{u\in C^k_{\loc}(I;X)\;:\;\|u\|_{C^k}<\infty\right\}.
\]
Moreover for $\alpha\in(0,1)$ we define the $C^{k,\alpha}$-norm by
\[
\|u\|_{C^{k,\alpha}}=\|u\|_{C^k}+\sup_{t\neq s\in I}\frac{\|\partial_t^ku(t)-\partial_t^ku(s)\|_X}{|t-s|^{\alpha}}\quad\mbox{for }u\in C^k_{\loc}(I;X),
\]
whenever it is finite. By $C^{k,\alpha}_{\loc}(I;X)$ we denote the space of maps $u\in C^{k}_{\loc}(I;X)$ with locally finite $C^{k,\alpha}$-norm, and we define
\[
C^{k,\alpha}(I;X)=\left\{u\in C^{k,\alpha}_{\loc}(I;X)\;:\;\|u\|_{C^{k,\alpha}}<\infty\right\}.
\]
Then $C^k(I;X)$ and $C^{k,\alpha}(I;X)$ are both Banach spaces.

Next we define Sobolev spaces of maps $u:I\rightarrow X$. Let $k\in\mathbb{N}$ and $p\in[1,\infty)$. For a $k$-times weakly differentiable map $u:I\rightarrow X$ we define the $W^{k,p}$-norm by 
\[
\|u\|_{W^{k,p}}=\left(\sum_{j=0}^k\int_I\|\partial_t^ju(t)\|_X^p\;\mathrm dt\right)^{1/p},
\]
whenever it is finite. We denote by $W^{k,p}_{\loc}(I;X)$ the space of $k$-times weakly differentiable maps $u:I\rightarrow X$ with locally finite $W^{k,p}$-norm, and we define
\[
W^{k,p}(I;X)=\left\{u\in W^{k,p}_{\loc}(I;X)\;:\;\|u\|_{W^{k,p}}<\infty\right\}.
\] 
Then $W^{k,p}(I;X)$ is a Banach space. If $k=0$, then we write $L^p_{\loc}(I;X)$ and $L^p(I;X)$ instead of $W^{0,p}_{\loc}(I;X)$ and $W^{0,p}(I;X)$, respectively.

We can now define weighted parabolic $C^k$-spaces and H\"older spaces. Let $(M,g)$ be a compact $m$-dimensional Riemannian manifold with conical singularities as in Definition \ref{DefinitionConicalSingularities}, $\rho$ a radius function on $M'$, and $I\subset\mathbb{R}$ an open and bounded interval. For $k,l\in\mathbb{N}$ with $2k\leq l$ and $\boldsymbol\gamma\in\mathbb{R}^n$ we define 
\[
C^{k,l}_{\boldsymbol\gamma}(I\times M')=\bigcap_{j=0}^kC^j(I;C^{l-2j}_{\boldsymbol\gamma-2j}(M')).
\]
If $\alpha\in(0,1)$, then we define the weighted parabolic H\"older space $C^{k,l,\alpha}_{\boldsymbol\gamma}(I\times M')$ by
\[
C^{k,l,\alpha}_{\boldsymbol\gamma}(I\times M')=\bigcap_{j=0}^kC^{j,\alpha/2}(I;C^{l-2j}_{\boldsymbol\gamma-2j}(M'))\cap C^{j}(I;C^{l-2j,\alpha}_{\boldsymbol\gamma-2j}(M')).
\]
Clearly $C^{k,l}_{\boldsymbol\gamma}(I\times M')$ and $C^{k,l,\alpha}_{\boldsymbol\gamma}(I\times M')$ are both Banach spaces.

Next we define weighted parabolic Sobolev spaces. Let $k,l\in\mathbb{N}$ with $2k\leq l$, $p\in[1,\infty)$, and $\boldsymbol\gamma\in\mathbb{R}^n$. The weighted parabolic Sobolev space $W^{k,l,p}_{\boldsymbol\gamma}(I\times M')$ is given by
\[
W^{k,l,p}_{\boldsymbol\gamma}(I\times M')=\bigcap_{j=0}^kW^{j,p}(I;W^{l-2j,p}_{\boldsymbol\gamma-2j}(M')).
\]
Then $W^{k,l,p}_{\boldsymbol\gamma}(I\times M')$ is a Banach space. 

Finally we define weighted parabolic spaces with discrete asymptotics. Thus if $m\geq 3$, then for $k,l\in\mathbb{N}$ with $2k\leq l$ we define the weighted parabolic $C^k$-space $C^{k,l}_{\boldsymbol\gamma,\mathsf{P}_{\boldsymbol\gamma}}(I\times M')$ with discrete asymptotics by
\[
C^{k,l}_{\boldsymbol\gamma,\mathsf{P}_{\boldsymbol\gamma}}(I\times M')=\bigcap_{j=0}^kC^j(I;C^{l-2j}_{\boldsymbol\gamma-2j,\mathsf{P}_{\boldsymbol\gamma-2j}}(M')),
\]
and if $\alpha\in(0,1)$, then we define the weighted parabolic H\"older space $C^{k,l,\alpha}_{\boldsymbol\gamma,\mathsf{P}_{\boldsymbol\gamma}}(I\times M')$ with discrete asymptotics by
\[
C^{k,l,\alpha}_{\boldsymbol\gamma,\mathsf{P}_{\boldsymbol\gamma}}(I\times M')=\bigcap_{j=0}^kC^{j,\alpha/2}(I;C^{l-2j}_{\boldsymbol\gamma-2j,\mathsf{P}_{\boldsymbol\gamma-2j}}(M'))\cap C^j(I;C^{l-2j,\alpha}_{\boldsymbol\gamma-2j,\mathsf{P}_{\boldsymbol\gamma-2j}}(M')).
\]
Then both $C^{k,l}_{\boldsymbol\gamma,\mathsf{P}_{\boldsymbol\gamma}}(I\times M')$ and $C^{k,l,\alpha}_{\boldsymbol\gamma,\mathsf{P}_{\boldsymbol\gamma}}(I\times M')$ are Banach spaces. If $p\in[1,\infty)$, then we define the weighted parabolic Sobolev space $W^{k,l,p}_{\boldsymbol\gamma,\mathsf{P}_{\boldsymbol\gamma}}(I\times M')$ with discrete asymptotics by 
\[
W^{k,l,p}_{\boldsymbol\gamma,\mathsf{P}_{\boldsymbol\gamma}}(I\times M')=\bigcap_{j=0}^kW^{j,p}(I;W^{l-2j,p}_{\boldsymbol\gamma-2j,\mathsf{P}_{\boldsymbol\gamma-2j}}(M')).
\]
Clearly $W^{k,l,p}_{\boldsymbol\gamma,\mathsf{P}_{\boldsymbol\gamma}}(I\times M')$ is a Banach space.

\subsection{Weighted Schauder and $L^p$-estimates}
The next proposition gives the weighted $L^p$-estimates for solutions to the inhomogeneous heat equation.

\begin{Prop}\label{WeightedLpEstimates}
Let $(M,g)$ be a compact $m$-dimensional Riemannian manifold with conical singularities as in Definition \ref{DefinitionConicalSingularities}. Let $T>0$, $k\in\mathbb{N}$ with $k\geq 2$, $p\in(1,\infty)$, and $\boldsymbol\gamma\in\mathbb{R}^n$. Let $f\in W^{0,k-2,p}_{\boldsymbol\gamma-2}((0,T)\times M')$ and $u\in W^{0,0,p}_{\boldsymbol\gamma}((0,T)\times M')$. Assume that $u\in W^{1,2}((0,T)\times N)$ for every $N\subset\subset M'$ and $\partial_tu=\Delta_gu+f$. Then $u\in W^{1,k,p}_{\boldsymbol\gamma}((0,T)\times M')$ and there exists a constant $c>0$ independent of $u$ and $f$, such that
\begin{equation}\label{LP1}
\|u\|_{W^{1,k,p}_{\boldsymbol\gamma}}\leq c\left(\|f\|_{W^{0,k-2,p}_{\boldsymbol\gamma-2}}+\|u\|_{W^{0,0,p}_{\boldsymbol\gamma}}\right).
\end{equation}
\end{Prop}

\begin{proof}
Let $f\in W^{0,k-2,p}_{\boldsymbol\gamma-2}((0,T)\times M')$ and assume that $u\in W^{0,0,p}_{\boldsymbol\gamma}((0,T)\times M')$ with $u\in W^{1,2,p}((0,T)\times N)$ for every $N\subset\subset M'$. Then it follows from the standard $L^p$-estimates \cite[Ch. 5, \S 2, Thm. 5]{KrylovSobolev} that for every $N_1,N_2\subset\subset M'$ with $N_1\subset\subset N_2$ there exists a constant $c>0$, such that
\begin{equation}\label{int}
\|u\|_{W^{1,k,p}}\leq c\left(\|f\|_{W^{0,k-2,p}}+\|u\|_{W^{0,0,p}}\right),
\end{equation}
where the norm on the left side is on $(0,T)\times N_1$ and the norm on the right side is on $(0,T)\times N_2$. Thus it remains to prove the weighted $L^p$-estimate \eq{LP1} near each singularity. Without loss of generality we can assume that $R\leq\sqrt{T}$, where $R$ is as in Definition \ref{DefinitionConicalSingularities}. Then for $s\in(0,R)$ and $i=1,\ldots,n$ we define 
\[
\delta_i^s:(\mbox{$\frac{1}{2}$},1)\times\Sigma_i\times(\mbox{$\frac{1}{2}$},1)\longrightarrow(0,T)\times\Sigma\times(0,R),\quad\delta_i^s(t,\sigma,r)=(s^2t,\sigma,sr).
\]
Denote $u_i=\phi_i^*(u)$ and $f_i=\phi_i^*(f)$ for $i=1,\ldots,n$ and define functions $u_i^s,f_i^s:\mbox{$(\frac{1}{2},1)\times\Sigma\times(\frac{1}{2},1)$}\rightarrow\mathbb{R}$ by $u_i^s=s^{-\gamma_i}(\delta_i^s)^*(u_i)$ and $f_i^s=s^{2-\gamma_i}(\delta_i^s)^*(f_i)$ for $s\in(0,R)$ and $i=1,\ldots,n$. Then there exists a constant $c>0$, such that 
\begin{equation}\label{JJJ}
\|u_i^s\|_{W^{0,0,p}},\|f_i^s\|_{W^{0,k-2,p}}\leq c\quad\mbox{on }\mbox{$(\frac{1}{2},1)\times\Sigma\times(\frac{1}{2},1)$}
\end{equation}
for $s\in(0,R)$ and $i=1,\ldots,n$. Using the definition of $u_i^s$ and $f_i^s$ we find that 
\[
\frac{\partial u_i^s}{\partial t}=\Delta_{g_i}u_i^s+L_i^su_i^s+f_i^s\quad\mbox{on }\mbox{$(\frac{1}{2},1)\times\Sigma\times(\frac{1}{2},1)$}
\]
for $i=1,\ldots,n$, where $L_i^s$ is a second order differential operator defined by
\[
L_i^sv=s^{2}\left\{\Delta_{\phi_i^*(g)}((\delta_i^{1/s})^*(v))-\Delta_{g_i}((\delta_i^{1/s})^*(v))\right\}\circ\delta_i^s.
\]
From \eq{ConicalSingularity} it follows that the coefficients of $L_i^s$ and their derivatives converge to zero uniformly on compact subsets of $\Sigma_i\times(\frac{1}{2},1)$ as $s\rightarrow 0$. Using \eq{JJJ} and again the $L^p$-estimates from \cite[Ch. 5, \S 2, Thm. 5]{KrylovSobolev} it follows that there exists a constant $c>0$, such that for every $s\in(0,\kappa)$, where $\kappa\in(0,R)$ is sufficiently small, and $i=1,\ldots,n$ we have
\begin{equation}\label{ext}
\|u_i^s\|_{W^{1,k,p}}\leq c\left(\|f_i^s\|_{W^{0,k-2,p}}+\|u_i^s\|_{W^{0,0,p}}\right),
\end{equation}
where the norm on the left side is on $(\frac{1}{2},1)\times\Sigma_i\times(\frac{2}{3},\frac{3}{4})$ and the norm on the right side is on $(\frac{1}{2},1)\times\Sigma_i\times(\frac{1}{2},1)$. Then it follows that $u\in W^{1,k,p}_{\boldsymbol\gamma}((0,T)\times M')$ and \eq{int} and \eq{ext} together imply \eq{LP1}.
\end{proof}

The next proposition gives the weighted Schauder estimates for solutions of the inhomogeneous heat equation.

\begin{Prop}\label{WeightedSchauderEstimates}
Let $(M,g)$ be a compact $m$-dimensional Riemannian manifold with conical singularities as in Definition \ref{DefinitionConicalSingularities}. Let $T>0$, $k\in\mathbb{N}$ with $k\geq 2$, $\alpha\in(0,1)$, and $\boldsymbol\gamma\in\mathbb{R}^n$. Let $f\in C^{0,k-2,\alpha}_{\boldsymbol\gamma-2}((0,T)\times M')$ and $u\in C^{0,0}_{\boldsymbol\gamma}((0,T)\times M')$. Assume that $u\in C^{1,k,\alpha}((0,T)\times N)$ for every $N\subset\subset M'$ and $\partial_tu=\Delta_gu+f$. Then $u\in C^{1,k,\alpha}_{\boldsymbol\gamma}((0,T)\times M')$ and there exists a constant $c>0$ independent of $u$ and $f$, such that
\begin{equation}\label{Schauder1}
\|u\|_{C^{1,k,\alpha}_{\boldsymbol\gamma}}\leq c\left(\|f\|_{C^{0,k-2,\alpha}_{\boldsymbol\gamma-2}}+\|u\|_{C^{0,0}_{\boldsymbol\gamma}}\right).
\end{equation}
\end{Prop}

\noindent Proposition \ref{WeightedSchauderEstimates} is proved in exactly the same way as Proposition \ref{WeightedLpEstimates}.

\subsection{The Cauchy problem for the heat equation. I. Maximal H\"older regularity}
Throughout this subsection $(M,g)$ will be a compact $m$-dimensional Riemannian manifold with conical singularities as in Definition \ref{DefinitionConicalSingularities}, $m\geq 3$, $\rho$ a radius function on $M'$, and $H$ will denote the Friedrichs heat kernel on $(M,g)$.

If $f:(0,T)\times M'\rightarrow\mathbb{R}$, $T>0$, is a function, then the convolution $H*f:(0,T)\times M'\rightarrow\mathbb{R}$ of $H$ and $f$ is given by
\begin{equation}\label{Convolution}
(H*f)(t,x)=\int_0^t\int_{M'}H(t-s,x,y)f(s,y)\;\mathrm dV_g(y)\;\mathrm ds
\end{equation}
for $t\in(0,T)$ and $x\in M'$, whenever it is well defined.

In the next two propositions we prove two elementary, though important, estimates for the convolution of $H$ with powers of $\rho$.

\begin{Prop}\label{EstimateONE}
Let $\boldsymbol\gamma\in\mathbb{R}^n$ with $\boldsymbol\gamma>2-m$ and $\gamma_i\notin\mathcal{E}_{\Sigma_i}$ for $i=1,\ldots,n$. Then there exist constants $c_l>0$ for $l\in\mathbb{N}$, such that
\[
\left|\left(\left(\partial_t^lH-\sum\nolimits_{j=1}^{N_l}\psi_{\boldsymbol\gamma-2l}^jH^j_{\boldsymbol\gamma-2l}\right)*\rho^{\boldsymbol\gamma-2}\right)(t,x)\right|\leq c_l\cdot\rho(x)^{\boldsymbol\gamma-2l}
\]
for every $t\in(0,\infty)$ and $x\in M'$. Here $\psi^j_{\boldsymbol\gamma-2l}$ and $H_{\boldsymbol\gamma-2l}^j$ are as in Theorem \ref{HeatKernelEstimates} for $j=1,\ldots,N_l$.
\end{Prop}

\begin{proof}
We only consider the case $l=0$, the general case is proved in a similar way. Throughout this proof $c$ will denote a positive constant that is independent of $t\in(0,\infty)$ and $x\in M'$ and that may be increased in each step of the proof. Denote
\[
I(t,x)=\left(\left(H-\sum\nolimits_{j=1}^{N_0}\psi^j_{\boldsymbol\gamma}H^j_{\boldsymbol\gamma}\right)*\rho^{\boldsymbol\gamma-2}\right)(t,x)
\]
for $t\in(0,\infty)$ and $x\in M'$. Using Theorem \ref{HeatKernelEstimates} we find that
\begin{align*}
&\left|I(t,x)\right|\leq\int_0^t\int_{M'}\left|H(s,x,y)-\sum\nolimits_{j=1}^{N_0}\psi^j_{\boldsymbol\gamma}(x)H^j_{\boldsymbol\gamma}(s,y)\right|\rho(y)^{\boldsymbol\gamma-2}\;\mathrm dV_g(y)\;\mathrm ds\\
&\;\;\;\leq c\cdot\rho(x)^{\boldsymbol\gamma^+}\int_{M'}\rho(y)^{\boldsymbol\gamma-2}(\rho(x)^2+\rho(y)^2)^{-\frac{\boldsymbol\gamma^+}{2}}\int_0^t(s+d_g(x,y)^2)^{-\frac{m}{2}}\;\mathrm ds\;\mathrm dV_g(y),
\end{align*}
where $\boldsymbol\gamma^+$ is as in \eq{BIPPELKISSEN}. Since $m\geq 3$, we can estimate the integral with respect to $s$ by
\[
\int_0^t(s+d_g(x,y)^2)^{-\frac{m}{2}}\;\mathrm ds\leq c\cdot d_g(x,y)^{2-m}
\]
and thus obtain
\begin{equation}\label{uu1}
|I(t,x)|\leq c\cdot \rho(x)^{\boldsymbol\gamma^+}\int_{M'}\rho(y)^{\boldsymbol\gamma-2}(\rho(x)^2+\rho(y)^2)^{-\frac{\boldsymbol\gamma^+}{2}}d_g(x,y)^{2-m}\;\mathrm dV_g(y).
\end{equation}
For the sake of simplicity we assume from now on that $\phi_i(g)=g_i$ for $i=1,\ldots,n$. The general case then follows in a similar way because the error terms caused by the asymptotic condition \eq{ConicalSingularity} can be controlled by the same estimates which we now prove. Let $R'>0$ with $\frac{R}{2}<R'<R$ and assume that $x$ lies in $S_i'=\{x\in M\;:\;0<d(x,x_i)<R'\}$ for some $i=1,\ldots,n$. The case $x\in M'\backslash S_i'$ is dealt with in a similar way. We now split the integral over $M'$ in \eq{uu1} into two integrals, one over $S_i$ and the other one over $M'\backslash S_i$. Recall that $S_i=\{x\in M\;:\;0<d(x,x_i)<R\}$. We first study the integral over $S_i$. 

If $y\in S_i$, then $d_g(x,y)^2\geq c(r^2+r'^2)d_h(\sigma,\sigma')^2$ for every $y\in S_i$ where $x=\phi_i(\sigma,r)$ and $y=\phi_i(\sigma',r')$. Moreover let us assume that $\rho(x)=r$ and $\rho(y)=r'$. Using $\mathrm dV_{g_i}(\sigma',r')=r'^{m-1}\mathrm dr'\;\mathrm dV_{h_i}(\sigma')$ we thus obtain
\begin{align*}
&\int_{S_i}\rho(y)^{\boldsymbol\gamma-2}(\rho(x)^2+\rho(y)^2)^{-\frac{\boldsymbol\gamma^+}{2}}d_g(x,y)^{2-m}\;\mathrm dV_g(y)\\&\quad\quad\quad\quad\leq c\int_0^R\int_{\Sigma_i}r'^{\gamma_i+m-3}(r^2+r'^2)^{1-\frac{m+\gamma_i^+}{2}}d_h(\sigma,\sigma')^{2-m}\;\mathrm dV_h(\sigma')\;\mathrm dr'\\&\quad\quad\quad\quad\leq
c\int_0^Rr'^{\gamma_i+m-3}(r^2+r'^2)^{1-\frac{m+\gamma_i^+}{2}}\;\mathrm dr',
\end{align*}
where in the last estimate we use that the integral with respect to $\sigma'$ is finite, since $\Sigma_i$ is compact and $\dim\Sigma_i=m-1$. With the change of variables $r'\mapsto\varrho=(\frac{r'}{r})^2$ we find
\[
\int_0^Rr'^{\gamma_i+m-3}(r^2+r'^2)^{1-\frac{m+\gamma_i^+}{2}}\;\mathrm dr'\leq c\cdot r^{\gamma_i-\gamma_i^+}\int_0^{\infty}\varrho^{\frac{\gamma_i+m-4}{2}}(1+\varrho)^{1-\frac{m+\gamma_i^+}{2}}\;\mathrm d\varrho.
\]
Now the integral with respect to $\varrho$ is finite if and only if $\frac{\gamma_i+m-4}{2}>-1$ and $\frac{\gamma_i+m-4}{2}+1-\frac{m+\gamma_i^+}{2}<-1$, which holds if and only if $2-m<\gamma_i<\gamma_i^+$. Therefore we obtain that
\begin{equation}\label{rte1}
\rho(x)^{\boldsymbol\gamma^+}\int_{S_i}\rho(y)^{\boldsymbol\gamma-2}(\rho(x)^2+\rho(y)^2)^{-\frac{\boldsymbol\gamma^+}{2}}d_g(x,y)^{2-m}\;\mathrm dV_g(y)\leq c\cdot\rho(x)^{\boldsymbol\gamma}.
\end{equation}

Now assume that $y\in M'\backslash S_i$. Then $d_g(x,y)$ is uniformly bounded from below, as $x\in S_i'$. Hence we can estimate $d_g(x,y)^2\geq c(\rho(x)^2+\rho(y)^2)$ uniformly for $y\in M'\backslash S_i$. From \eq{uu1} we thus obtain
\begin{align*}
&\int_{M'\backslash S_i}\rho(y)^{\boldsymbol\gamma-2}(\rho(x)^2+\rho(y)^2)^{-\frac{\boldsymbol\gamma^+}{2}}d_g(x,y)^{2-m}\;\mathrm dV_g(y)\\&\quad\quad\quad\quad\leq c\int_{M'\backslash S_i}\rho(y)^{\boldsymbol\gamma-2}(\rho(x)^2+\rho(y)^2)^{1-\frac{m+\boldsymbol\gamma^+}{2}}\;\mathrm dV_g(y).
\end{align*}
Using the same estimates as before it is now straightforward to check that
\[
\int_{M'\backslash S_i}\rho(y)^{\boldsymbol\gamma-2}(\rho(x)^2+\rho(y)^2)^{1-\frac{m+\boldsymbol\gamma^+}{2}}\;\mathrm dV_g(y)\leq c\cdot\rho(x)^{\boldsymbol\gamma-\boldsymbol\gamma^+}.
\]
We find that
\begin{equation}\label{rte2}
\rho(x)^{\boldsymbol\gamma^+}\int_{M'\backslash S_i}\rho(y)^{\boldsymbol\gamma-2}(\rho(x)^2+\rho(y)^2)^{-\frac{\boldsymbol\gamma^+}{2}}d_g(x,y)^{2-m}\;\mathrm dV_g(y)\leq c\cdot\rho(x)^{\boldsymbol\gamma}.
\end{equation}

Finally from \eq{uu1}, \eq{rte1}, and \eq{rte2} we conclude that $|I(t,x)|\leq c\cdot\rho(x)^{\boldsymbol\gamma}$ for $t\in(0,\infty)$ and $x\in M'$, as we wanted to show.
\end{proof}

\begin{Prop}\label{EstimateTWO}
Let $\boldsymbol\gamma\in\mathbb{R}^n$ with $\boldsymbol\gamma>2-m$ and $\gamma_i\notin\mathcal{E}_{\Sigma_i}$ for $i=1,\ldots,n$. Then there exist constants $c_l>0$ for $l\in\mathbb{N}$, such that $\big|(H^j_{\boldsymbol\gamma-2l}*\rho^{\boldsymbol\gamma-2})(t)\big|\leq c_l$ for every $t\in(0,\infty)$ and $j=1,\ldots,N_l$. Here $H^j_{\boldsymbol\gamma-2l}$ is as in Theorem \ref{HeatKernelEstimates}.
\end{Prop}

\begin{proof}
Again we only consider the case $l=0$. Again $c$ will denote a positive constant that is independent of $t\in(0,\infty)$ and $x\in M'$ and that may be increased in each step of the proof. Fix some $j=1,\ldots,N_0$ and denote $I(t)=(H^j_{\boldsymbol\gamma}*\rho^{\boldsymbol\gamma-2})(t)$ for $t\in(0,\infty)$. Using the estimates from Theorem \ref{HeatKernelEstimates} we obtain that
\begin{align*}
|I(t)|&\leq\int_{M'}|H^j_{\boldsymbol\gamma}(s,y)|\rho(y)^{\boldsymbol\gamma-2}\;\mathrm dV_{g}(y)\\
&\leq c\int_{M'}\rho(y)^{\boldsymbol\gamma-2}\int_0^t(s+\rho(y)^2)^{-\frac{m+\boldsymbol\gamma^-}{2}}\;\mathrm ds\;\mathrm dV_g(y),
\end{align*}
where $\boldsymbol\gamma^-$ is as in \eq{BIPPELKISSEN}. We can estimate the integral with respect to $s$ by
\[
\int_0^t(s+\rho(y)^2)^{-\frac{m+\boldsymbol\gamma^-}{2}}\;\mathrm ds\leq c\cdot\rho(y)^{2-m-\boldsymbol\gamma^-}
\]
and hence we obtain that
\begin{equation}\label{rte3}
|I(t)|\leq c\int_{M'}\rho(y)^{\boldsymbol\gamma-\boldsymbol\gamma^--m}\;\mathrm dV_g(y).
\end{equation}
Using that the Riemannian metric $\phi_i^*(g)$ is asymptotic to the Riemannian cone metric $g_i$ on $\Sigma_i\times(0,R)$ and using that $\mathrm dV_{g_i}(\sigma,r)=r^{m-1}\mathrm dr\;\mathrm dV_{h_i}(\sigma)$ it follows that the integral in \eq{rte3} is finite if and only if $\boldsymbol\gamma-\boldsymbol\gamma^--m+(m-1)>-1$, which holds if and only if $\boldsymbol\gamma>\boldsymbol\gamma^-$.
\end{proof}

Using Propositions \ref{EstimateONE} and \ref{EstimateTWO} we are now able to prove existence and maximal regularity of solutions to \eq{HeatEquation}, when $f$ lies in a weighted parabolic H\"older space with discrete asymptotics.

\begin{Thm}\label{HoelderRegularityHeat}
Let $(M,g)$ be a compact $m$-dimensional Riemannian manifold with conical singularities as in Definition \ref{DefinitionConicalSingularities}, $m\geq 3$. Let $T>0$, $k\in\mathbb{N}$ with $k\geq 2$, $\alpha\in(0,1)$, and $\boldsymbol\gamma\in\mathbb{R}^n$ with $\boldsymbol\gamma>2-m$ and $\gamma_i\notin\mathcal{E}_{\Sigma_i}$ for $i=1,\ldots,n$. Given $f\in C^{0,k-2,\alpha}_{\boldsymbol\gamma-2,\mathsf{P}_{\boldsymbol\gamma-2}}((0,T)\times M')$, there exists a unique $u\in C^{1,k,\alpha}_{\boldsymbol\gamma,\mathsf{P}_{\boldsymbol\gamma}}((0,T)\times M')$ solving the Cauchy problem \eq{HeatEquation}.
\end{Thm}

\begin{proof}
Let $f\in C^{0,k-2,\alpha}_{\boldsymbol\gamma-2,\mathsf{P}_{\boldsymbol\gamma-2}}((0,T)\times M')$, then we can write $f=f_1+f_2$ with 
\[
f_1\in C^{0,k-2,\alpha}_{\boldsymbol\gamma-2}((0,T)\times M')\quad\mbox{and}\quad f_2\in C^{0,\alpha/2}((0,T);\im\;\Psi_{\boldsymbol\gamma-2}).
\]
Define $u=H*f$, $u_1=H*f_1$, and $u_2=H*f_2$, where $H$ is the Friedrichs heat kernel. Using Theorem \ref{HeatKernelEstimates} we can write
\begin{align*}
u_1(t,x)=\left(\left(H-\sum\nolimits_{j=1}^{N_0}\psi^j_{\boldsymbol\gamma}H^j_{\boldsymbol\gamma}\right)*f_1\right)(t,x)+\sum\nolimits_{j=1}^{N_0}\psi^j_{\boldsymbol\gamma}(x)(H^j_{\boldsymbol\gamma}*f_1)(t)
\end{align*}
for $t\in(0,T)$ and $x\in M'$. Using that $f_1\in C^{0,k-2,\alpha}_{\boldsymbol\gamma-2}((0,T)\times M')$ and Proposition \ref{EstimateONE} we find that
\[
\left|\left(\left(H-\sum\nolimits_{j=1}^{N_0}\psi^j_{\boldsymbol\gamma}H^j_{\boldsymbol\gamma}\right)*f_1\right)(t,x)\right|\leq c\|f_1\|_{C^{0,0}_{\boldsymbol\gamma-2}}\rho(x)^{\boldsymbol\gamma}.
\]
Moreover, from Proposition \ref{EstimateTWO} it follows that 
\[
|(H^j_{\boldsymbol\gamma}*f_1)(t)|\leq c\|f_1\|_{C^{0,0}_{\boldsymbol\gamma-2}}
\]
for $j=1,\ldots,N_0$. Hence $u_1\in C^{0,0}_{\boldsymbol\gamma,\mathsf{P}_{\boldsymbol\gamma}}((0,T)\times M')$. In a similar way one can now show that in fact $u_1\in C^{1,k,\alpha}_{\boldsymbol\gamma,\mathsf{P}_{\boldsymbol\gamma}}((0,T)\times M')$. 

Alternatively one can show that $u_1\in C^{1,k,\alpha}((0,T)\times N)$ for every $N\subset\subset M'$. If $\boldsymbol\gamma <0$, then the discrete asymptotics are trivial and the weighted Schauder estimates from Proposition \ref{WeightedSchauderEstimates} imply that $u_1\in C^{1,k,\alpha}_{\boldsymbol\gamma,\mathsf{P}_{\boldsymbol\gamma}}((0,T)\times M')$. If $\boldsymbol\gamma>0$, then $u_1\in C^{0,0}_{\boldsymbol 0}((0,T)\times M')$, since the discrete asymptotics are bounded functions on $M'$. Therefore again the weighted Schauder estimates imply that in fact $u_1\in C^{1,k,\alpha}_{\boldsymbol 0}((0,T)\times M')$. But then using Proposition \ref{Invertibility} and a simple iteration argument we conclude that $u_1\in C^{1,k,\alpha}_{\boldsymbol\gamma,\mathsf{P}_{\boldsymbol\gamma}}((0,T)\times M')$. 

The same argument as before also shows that $u_2=H*f_2$ lies in $C^{1,l,\alpha}_{\boldsymbol\delta,\mathsf{P}_{\boldsymbol\delta}}((0,T)\times M')$ for every $l\in\mathbb{N}$ and $\boldsymbol\delta\in\mathbb{R}^n$. Hence $u\in C^{1,k,\alpha}_{\boldsymbol\gamma,\mathsf{P}_{\boldsymbol\gamma}}((0,T)\times M')$ and $u$ solves the Cauchy problem \eq{HeatEquation}. 

In order to show that $u$ is the unique solution of the Cauchy problem \eq{HeatEquation} it suffices to show that if $u\in C^{1,k,\alpha}_{\boldsymbol\gamma,\mathsf{P}_{\boldsymbol\gamma}}((0,T)\times M')$ solves the Cauchy problem \eq{HeatEquation} with $f\equiv 0$, then $u\equiv 0$. Thus let $u\in C^{1,k,\alpha}_{\boldsymbol\gamma,\mathsf{P}_{\boldsymbol\gamma}}((0,T)\times M')$ be a solution of \eq{HeatEquation} with $f\equiv 0$ and assume first that $\boldsymbol\gamma>3-\frac{m}{2}$. Then for $t\in(0,T)$
\[
\frac{\mathrm d}{\mathrm dt}\|u(t,\cdot)\|_{L^2}^2=2\langle\Delta_gu(t,\cdot),u(t,\cdot)\rangle_{L^2}=-2\|\nabla u(t,\cdot)\|_{L^2}^2\leq 0.
\] 
Since $u(0,\cdot)\equiv 0$, it follows that $u\equiv 0$. Now assume that $\boldsymbol\gamma\in\mathbb{R}^n$ with $2-m<\boldsymbol\gamma\leq 3-\frac{m}{2}$. Then it easily follows that $\int_{M'}u(t,x)\;\mathrm dV_g(x)=0$ for $t\in(0,T)$. Using Proposition \ref{Invertibility} we can define $u_1=\Delta_g^{-1}u$. Then $u_1\in C^{1,k+2,\alpha}_{\boldsymbol\gamma+2,\mathsf{P}_{\boldsymbol\gamma+2}}((0,T)\times M')$ and $u_1$ solves the Cauchy problem \eq{HeatEquation} with $f\equiv 0$. We can iterate this argument and define $u_l=\Delta_g^{-l}u$ for $l\in\mathbb{N}$ with $\boldsymbol\gamma+2l>3-\frac{m}{2}$. Then $u_l\in C^{1,k+2l,\alpha}_{\boldsymbol\gamma+2l,\mathsf{P}_{\boldsymbol\gamma+2l}}((0,T)\times M')$ and $u_l$ solves the Cauchy problem \eq{HeatEquation} with $f\equiv 0$. Then as above it follows that $u_l\equiv 0$ and hence $u\equiv 0$. This completes the proof of Theorem \ref{HoelderRegularityHeat}.
\end{proof}

\subsection{The Cauchy problem for the heat equation. II. Maximal Sobolev regularity}
We first prove a generalization of Young's inequality on Riemannian manifolds with conical singularities involving weighted $L^p$-norms. The proof follows the same ideas as in Aubin \cite[Prop. 3.64]{Aubin}.

\begin{Prop}\label{Conv}
Let $(M,g)$ be a compact $m$-dimensional Riemannian manifold with conical singularities as in Definition \ref{DefinitionConicalSingularities}, $T>0$, $p\in(1,\infty)$, and $\boldsymbol\delta,\boldsymbol\varepsilon\in\mathbb{R}^n$. Let $f\in W^{0,0,p}_{\boldsymbol\varepsilon}((0,T)\times M')$ and $G\in C^0_{\loc}(((0,T)\times(0,T)\times M'\times M')\backslash\Delta)$, where $\Delta=\{(t,t,x,x)\;:\;t\in(0,T),\;x\in M'\}$. Assume that
\[
\sup_{\substack{t\in(0,T) \\ x\in M'}}\rho(x)^{\boldsymbol\alpha_2}\|G(t,\cdot,x,\cdot)\|_{W^{0,0,1}_{-\boldsymbol\beta_2-m}},\;\sup_{\substack{s\in(0,T) \\ y\in M'}}\rho(y)^{\boldsymbol\beta_1}\|G(\cdot,s,\cdot,y)\|_{W^{0,0,1}_{-\boldsymbol\alpha_1+\boldsymbol\delta p}}<\infty
\]
for some $\boldsymbol\alpha_1,\boldsymbol\alpha_2,\boldsymbol\beta_1,\boldsymbol\beta_2\in\mathbb{R}^n$ that satisfy
\begin{equation}\label{Relations}
\frac{\boldsymbol\alpha_1}{p}+\boldsymbol\alpha_2\left(1-\frac{1}{p}\right)=0\quad\mbox{and}\quad\frac{\boldsymbol\beta_1}{p}+\boldsymbol\beta_2\left(1-\frac{1}{p}\right)=\boldsymbol\varepsilon+\frac{m}{p}.
\end{equation}
Then $G*f\in W^{0,0,p}_{\boldsymbol\delta}((0,T)\times M')$ and moreover 
\begin{eqnarray*}
&&\|G*f\|_{W^{0,0,p}_{\boldsymbol\delta}}\leq\|f\|_{W^{0,0,p}_{\boldsymbol\varepsilon}}\sup_{\substack{t\in(0,T) \\ x\in M'}}\rho(x)^{\boldsymbol\alpha_2(1-\frac{1}{p})}\|G(t,\cdot,x,\cdot)\|^{1-\frac{1}{p}}_{W^{0,0,1}_{-\boldsymbol\beta_2-m}}\\&&\quad\quad\quad\quad\quad\quad\quad\quad\quad\quad\quad\quad\times\sup_{\substack{s\in(0,T) \\ y\in M'}}\rho(y)^{\frac{\boldsymbol\beta_1}{p}}\|G(\cdot,s,\cdot,y)\|^{\frac{1}{p}}_{W^{0,0,1}_{-\boldsymbol\alpha_1+\boldsymbol\delta p}}.
\end{eqnarray*}
\end{Prop}

\begin{proof}
Without loss of generality we can assume that $f$ and $G$ are non-negative.
We write
\begin{align*}
&G(t,s,x,y)f(s,y)=\left(G(t,s,x,y)f(s,y)^p\right)^{\frac{1}{p}}G(t,s,x,y)^{1-\frac{1}{p}}\\
&\quad\quad=\left(G(t,s,x,y)f_{\boldsymbol\varepsilon}(s,y)^p\right)^{\frac{1}{p}}G(t,s,x,y)^{1-\frac{1}{p}}\rho(y)^{\boldsymbol\varepsilon+\frac{m}{p}}\\
&\quad\quad=\left(\rho(x)^{\boldsymbol\alpha_1}\rho(y)^{\boldsymbol\beta_1}G(t,s,x,y)f_{\boldsymbol\varepsilon}(s,y)^p\right)^{\frac{1}{p}}\left(\rho(x)^{\boldsymbol\alpha_2}\rho(y)^{\boldsymbol\beta_2}G(t,s,x,y)\right)^{1-\frac{1}{p}},
\end{align*}
where $f_{\boldsymbol\varepsilon}(s,y)=\rho(y)^{-\boldsymbol\varepsilon-\frac{m}{p}}f(s,y)$ and $\boldsymbol\alpha_1,\boldsymbol\alpha_2,\boldsymbol\beta_1,\boldsymbol\beta_2\in\mathbb{R}^n$ satisfy \eq{Relations}. Using H\"older's inequality we find
\begin{align*}
|(G*f)(t,x)|&\leq\left(\int_0^T\int_{M'}\rho(x)^{\boldsymbol\alpha_1}\rho(y)^{\boldsymbol\beta_1}G(t,s,x,y)f_{\boldsymbol\varepsilon}(s,y)^p\mathrm dV_g(y)\;\mathrm ds\right)^{\frac{1}{p}}\\&\quad\quad\times\left(\int_0^T\int_{M'}\rho(x)^{\boldsymbol\alpha_2}\rho(y)^{\boldsymbol\beta_2}G(t,s,x,y)\;\mathrm dV_g(y)\;\mathrm ds\right)^{1-\frac{1}{p}}.
\end{align*}
It follows that
\begin{align*}
&\|G*f\|_{W^{0,0,p}_{\boldsymbol\delta}}^p\leq\int\limits_0^T\int\limits_{M'}\left\{\int\limits_0^T\int\limits_{M'}\rho(x)^{\boldsymbol\zeta_1}\rho(y)^{\boldsymbol\beta_1}G(t,s,x,y)f_{\boldsymbol\varepsilon}(s,y)^p\;\mathrm dV_g(y)\;\mathrm ds\right.\\&\left.\quad\quad\quad\quad\times\left(\int\limits_0^T\int\limits_{M'}\rho(x)^{\boldsymbol\alpha_2}\rho(y)^{\boldsymbol\beta_2}G(t,s,x,y)\;\mathrm dV_g(y)\;\mathrm ds\right)^{p-1}\right\}\;\mathrm dV_g(x)\;\mathrm dt
\end{align*}
with $\boldsymbol\zeta_1=\boldsymbol\alpha_1-\boldsymbol\delta p-m$. Observe that
\begin{align*}
\int_0^T\int_{M'}\rho(x)^{\boldsymbol\alpha_2}\rho(y)^{\boldsymbol\beta_2}G(t,s,x,y)\;\mathrm dV_g(y)\;\mathrm ds=\rho(x)^{\boldsymbol\alpha_2}\|G(t,\cdot,x,\cdot)\|_{W^{0,0,1}_{-\boldsymbol\beta_2-m}}.
\end{align*}
Hence
\begin{align*}
&\|G*f\|_{W^{0,0,p}_{\boldsymbol\delta}}^p\leq\left(\sup_{\substack{t\in(0,T) \\ x\in M'}}\rho(x)^{\boldsymbol\alpha_2}\|G(t,\cdot,x,\cdot)\|_{W^{0,0,1}_{-\boldsymbol\beta_2-m}}\right)^{p-1}\quad\\&\times\int\limits_0^T\int\limits_{M'}\left\{\int\limits_0^T\int\limits_M\rho(x)^{\boldsymbol\zeta_1}\rho(y)^{\boldsymbol\beta_1}G(t,s,x,y)f_{\boldsymbol\varepsilon}(s,y)^p\;\mathrm dV_g(y)\;\mathrm ds\right\}\;\mathrm dV_g(x)\;\mathrm dt.
\end{align*}
Finally we have 
\begin{align*}
&\int\limits_0^T\int\limits_{M'}\left\{\int\limits_0^T\int\limits_{M'}\rho(x)^{\boldsymbol\zeta_1}\rho(y)^{\boldsymbol\beta_1}G(t,s,x,y)f_{\boldsymbol\varepsilon}(s,y)^p\;\mathrm dV_g(y)\;\mathrm ds\right\}\;\mathrm dV_g(x)\;\mathrm dt\quad\\&\quad=\int\limits_0^T\int\limits_{M'}\left\{\int\limits_0^T\int\limits_{M'}\rho(x)^{\boldsymbol\zeta_1}G(t,s,x,y)\;\mathrm dV_g(x)\;\mathrm dt\right\}\rho(y)^{\boldsymbol\beta_1}f_{\boldsymbol\varepsilon}(s,y)^p\;\mathrm dV_g(y)\;\mathrm ds\\&\quad\quad=\int\limits_0^T\int\limits_{M'}\|G(\cdot,s,\cdot,y)\|_{W^{0,0,1}_{-\boldsymbol\alpha_1+\boldsymbol\delta p}}\rho(y)^{\boldsymbol\beta_1}f_{\boldsymbol\varepsilon}(s,y)^p\;\mathrm dV_g(y)\;\mathrm ds\\&\quad\quad\quad\leq\|f\|_{\boldsymbol\varepsilon}^p\sup_{\substack{s\in(0,T) \\ y\in M'}}\rho(y)^{\boldsymbol\beta_1}\|G(\cdot,s,\cdot,y)\|_{W^{0,0,1}_{-\boldsymbol\alpha_1+\boldsymbol\delta p}}
\end{align*}
from which the claim follows.
\end{proof}

The next proposition is proved in a similar way to Proposition \ref{Conv}.

\begin{Prop}\label{Conv2}
Let $(M,g)$ be a compact $m$-dimensional Riemannian manifold with conical singularities as in Definition \ref{DefinitionConicalSingularities}, $T>0$, $p\in(1,\infty)$, and $\boldsymbol\delta,\boldsymbol\varepsilon\in\mathbb{R}^n$. Let $f\in W^{0,0,p}_{\boldsymbol\varepsilon}((0,T)\times M')$ and $G\in C^0_{\loc}(((0,T)\times(0,T)\times M')\backslash\Delta)$, where $\Delta=\{(t,t,x)\;:\;t\in(0,T),\;x\in M'\}$. Assume that
\[
\sup_{t\in(0,T)}\|G(t,\cdot,\cdot)\|_{W^{0,0,1}_{-\boldsymbol\alpha_2-m}},\;\sup_{\substack{s\in(0,T) \\ x\in M'}}\rho(x)^{\boldsymbol\alpha_1}\|G(\cdot,s,y)\|_{L^1}<\infty
\]
for some $\boldsymbol\alpha_1,\boldsymbol\alpha_2\in\mathbb{R}^n$ that satisfy
\begin{equation}\label{Relations2}
\frac{\boldsymbol\alpha_1}{p}+\boldsymbol\alpha_2\left(1-\frac{1}{p}\right)=\boldsymbol\varepsilon+\frac{m}{p}.
\end{equation}
Then $G*f\in L^p((0,T))$ and moreover 
\begin{eqnarray*}
&&\|G*f\|_{L^p}\leq\|f\|_{W^{0,0,p}_{\boldsymbol\varepsilon}}\sup_{t\in(0,T)}\|G(t,\cdot,\cdot)\|^{1-\frac{1}{p}}_{W^{0,0,1}_{-\boldsymbol\alpha_2-m}}\sup_{\substack{s\in(0,T) \\ x\in M'}}\rho(x)^{\frac{\boldsymbol\alpha_1}{p}}\|G(\cdot,s,x)\|^{\frac{1}{p}}_{L^1}.
\end{eqnarray*}
\end{Prop}

Using Propositions \ref{Conv} and \ref{Conv2} we are now able to prove existence and maximal regularity of solutions to the Cauchy problem \eq{HeatEquation}, when $f$ lies in a weighted parabolic Sobolev space with discrete asymptotics.

\begin{Thm}\label{SobolevRegularityHeat}
Let $(M,g)$ be a compact $m$-dimensional Riemannian manifold with conical singularities as in Definition \ref{DefinitionConicalSingularities}, $m\geq 3$. Let $T>0$, $k\in\mathbb{N}$ with $k\geq 2$, $p\in(1,\infty)$, and $\boldsymbol\gamma\in\mathbb{R}^n$ with $\boldsymbol\gamma>2-m$ and $\gamma_i\notin\mathcal{E}_{\Sigma_i}$ for $i=1,\ldots,n$. Given $f\in W^{0,k-2,p}_{\boldsymbol\gamma-2,\mathsf{P}_{\boldsymbol\gamma-2}}((0,T)\times M')$, then there exists a unique $u\in W^{1,k,p}_{\boldsymbol\gamma,\mathsf{P}_{\boldsymbol\gamma}}((0,T)\times M')$ solving the Cauchy problem \eq{HeatEquation}.
\end{Thm}

\begin{proof}
Let $f\in W^{0,k-2,p}_{\boldsymbol\gamma-2,\mathsf{P}_{\boldsymbol\gamma-2}}((0,T)\times M')$. Then we can write $f=f_1+f_2$ with
\[
f_1\in W^{0,k-2,p}_{\boldsymbol\gamma-2}((0,T)\times M')\quad\mbox{and}\quad f_2\in L^p((0,T);\im\;\Psi_{\boldsymbol\gamma-2}).
\]
Let $H$ be the Friedrichs heat kernel on $(M,g)$ and define $u=H*f$, $u_1=H*f_1$, and $u_2=H*f_2$.

The first step is to show that $u_1\in W^{0,0,p}_{\boldsymbol\gamma,\mathsf{P}_{\boldsymbol\gamma}}((0,T)\times M)$. Using Theorem \ref{HeatKernelEstimates} we write
\begin{equation}\label{h1}
u_1(t,x)=\left(\left(H-\sum\nolimits_{j=1}^{N_0}\psi^j_{\boldsymbol\gamma}H^j_{\boldsymbol\gamma}\right)*f_1\right)(t,x)+\sum\nolimits_{j=1}^{N_0}\psi^j_{\boldsymbol\gamma}(x)(H^j_{\boldsymbol\gamma}*f_1)(t)
\end{equation}
for $t\in(0,T)$ and $x\in M'$. We begin by showing that the first term on the right side of \eq{h1} lies in $W^{0,0,p}_{\boldsymbol\gamma}((0,T)\times M')$. Define $G\in C^0_{\loc}(((0,T)\times(0,T)\times M'\times M')\backslash\Delta)$ by
\[
G(t,s,x,y)=H(|t-s|,x,y)-\sum\nolimits_{j=1}^{N_0}\psi^j_{\boldsymbol\gamma}(x)H^j_{\boldsymbol\gamma}(|t-s|,y),
\]
where $\Delta$ is as in Proposition \ref{Conv}. Notice that 
\[
\left|\left(H-\sum\nolimits_{j=1}^{N_0}\psi^j_{\boldsymbol\gamma}H^j_{\boldsymbol\gamma}\right)*f_1(t,x)\right|\leq (|G|*|f_1|)(t,x).
\]
We now apply Proposition \ref{Conv} with $\boldsymbol\delta=\boldsymbol\gamma$ and $\boldsymbol\varepsilon=\boldsymbol\gamma-2$. Then we have to show that that
\[
\sup_{\substack{t\in(0,T)\\x\in M'}}\rho(x)^{\boldsymbol\alpha_2}\|G(t,\cdot,x,\cdot)\|_{W^{0,0,1}_{-\boldsymbol\beta_2-m}},\;\sup_{\substack{s\in(0,T)\\ y\in M'}}\rho(y)^{\boldsymbol\beta_1}\|G(\cdot,s,\cdot,y)\|_{W^{0,0,1}_{-\boldsymbol\alpha_1+\boldsymbol\gamma p}}<\infty,
\]
where $\boldsymbol\alpha_1,\boldsymbol\alpha_2,\boldsymbol\beta_1,\boldsymbol\beta_2\in\mathbb{R}^n$ satisfy \eq{Relations}. Since $\boldsymbol\gamma^+\geq 0$, where $\boldsymbol\gamma^+$ is as in \eq{BIPPELKISSEN}, and $p>1$, it suffices to prove that 
\begin{equation}\label{E1}
\sup_{\substack{t\in(0,T)\\x\in M'}}\rho(x)^{\boldsymbol\alpha_2}\|G(t,\cdot,x,\cdot)\|_{W^{0,0,1}_{-\boldsymbol\beta_2-m+\frac{\boldsymbol\gamma^+}{p-1}}},\;\sup_{\substack{s\in(0,T)\\ y\in M'}}\rho(y)^{\boldsymbol\beta_1}\|G(\cdot,s,\cdot,y)\|_{W^{0,0,1}_{-\boldsymbol\alpha_1+\boldsymbol\gamma p}}<\infty.
\end{equation}
We analyze the first term. Note that
\[
\|G(t,\cdot,x,\cdot)\|_{W^{0,0,1}_{-\boldsymbol\beta_2-m+\frac{\boldsymbol\gamma^+}{p-1}}}=|G*\rho^{\boldsymbol\beta_2-\frac{\boldsymbol\gamma^+}{p-1}}|(t,x).
\]
If $-m<\boldsymbol\beta_2-\frac{\boldsymbol\gamma^+}{p-1}<\boldsymbol\gamma^+-2$, then by Proposition \ref{EstimateONE} there exists a constant $c>0$, such that 
\[
|G*\rho^{\boldsymbol\beta_2-\frac{\boldsymbol\gamma^+}{p-1}}|(t,x)\leq c\cdot \rho(x)^{\boldsymbol\beta_2-\frac{\boldsymbol\gamma^+}{p-1}+2}.
\]
Hence, if
\begin{equation}\label{REL3}
\boldsymbol\alpha_2+\boldsymbol\beta_2-\frac{\boldsymbol\gamma^+}{p-1}+2\geq 0\quad\mbox{and}\quad-m<\boldsymbol\beta_2-\frac{\boldsymbol\gamma^+}{p-1}<\boldsymbol\gamma^+-2,
\end{equation}
then the first term in \eq{E1} is finite. In a similar way we find that if 
\begin{equation}\label{REL4}
-m<\boldsymbol\alpha_1-\boldsymbol\gamma p-m<2-\boldsymbol\gamma^+\quad\mbox{and}\quad\boldsymbol\beta_1+\boldsymbol\alpha_1-\boldsymbol\gamma p-m+2\geq 0,
\end{equation}
then the second term in \eq{E1} is finite. A straightforward computation now shows that \eq{REL3} and \eq{REL4} are equivalent to the existence of a $\boldsymbol\beta\in\mathbb{R}^n$ with
\begin{equation*}
\frac{\boldsymbol\gamma^+}{p-1}-m<\boldsymbol\beta<\frac{\boldsymbol\gamma^+}{p-1}+\boldsymbol\gamma^+-2\quad\mbox{and}\quad\frac{\boldsymbol\gamma p}{p-1}-2<\boldsymbol\beta<\frac{\boldsymbol\gamma+2-m+\boldsymbol\gamma^+}{p-1}-2.
\end{equation*}
Such a $\boldsymbol\beta$ exists if and only if $m\geq 3$ and $2-m<\boldsymbol\gamma<\boldsymbol\gamma^+$. It follows that
\begin{equation}\label{TIGGERHANDPUEPPCHEN}
\left(H-\sum\nolimits_{j=1}^{N_0}\psi^j_{\boldsymbol\gamma}H^j_{\boldsymbol\gamma}\right)*f_1\in W^{0,0,p}_{\boldsymbol\gamma}((0,T)\times M').
\end{equation} 

The next step is to show that $H^j_{\boldsymbol\gamma}*f_1\in L^{p}((0,T))$ for $j=1,\ldots,N_0$. Fix some $j=1,\ldots,N_0$ and define $G\in C^0_{\loc}(((0,T)\times(0,T)\times M')\backslash\Delta)$ by $G(t,s,x)=H^j_{\boldsymbol\gamma}(|t-s|,x)$, where $\Delta$ is as in Proposition \ref{Conv2}. We now apply Proposition \ref{Conv2} with $\boldsymbol\varepsilon=\boldsymbol\gamma-2$. Then it suffices to show that
\begin{equation}\label{h2}
\sup_{t\in(0,T)}\|G(t,\cdot,\cdot)\|_{W^{0,0,1}_{-\boldsymbol\alpha_2-m}},\;\sup_{\substack{s\in(0,T) \\ x\in M'}}\rho(x)^{\boldsymbol\alpha_1}\|G(\cdot,s,y)\|_{L^1}<\infty
\end{equation}
for some $\boldsymbol\alpha_1,\boldsymbol\alpha_2\in\mathbb{R}^n$ that satisfy \eq{Relations2}. Using Proposition \ref{EstimateTWO} it follows that if 
\begin{equation}\label{K2}
\boldsymbol\alpha_2>\boldsymbol\gamma^-+2\quad\mbox{and}\quad\boldsymbol\alpha_1+2-m-\boldsymbol\gamma^-\geq0,
\end{equation}
then the two terms in \eq{h2} are finite. It is now straightforward to show that \eq{Relations2} and \eq{K2} have a solution if and only if $\boldsymbol\gamma>\boldsymbol\gamma^-$. Together with \eq{TIGGERHANDPUEPPCHEN} we conclude that $u_1\in W^{0,0,p}_{\boldsymbol\gamma,\mathsf{P}_{\boldsymbol\gamma}}((0,T)\times M')$. 

The same arguments as in the proof of Theorem \ref{HoelderRegularityHeat} then show that in fact $u_1\in W^{1,k,p}_{\boldsymbol\gamma,\mathsf{P}_{\boldsymbol\gamma}}((0,T)\times M')$ and $u_2\in W^{1,l,p}_{\boldsymbol\delta,\mathsf{P}_{\boldsymbol\delta}}((0,T)\times M')$ for every $l\in\mathbb{N}$ and $\boldsymbol\delta\in\mathbb{R}^n$. Hence $u\in W^{1,k,p}_{\boldsymbol\gamma,\mathsf{P}_{\boldsymbol\gamma}}((0,T)\times M')$ as we wanted to show. Finally the uniqueness follows as in the proof of Theorem \ref{HoelderRegularityHeat}.
\end{proof}

{\footnotesize\noindent LINCOLN COLLEGE, TURL STREET, OX1 3DR, OXFORD, UNITED KINGDOM}\\
{\small\verb|tapio.behrndt@gmail.com|}


\begin{thebibliography}{99}
\bibitem{Aubin}T. Aubin, \textit{Some Nonlinear Problems in Riemannian Geometry}, Springer-Verlag, Heidelberg, 1998.
\bibitem{Bartnik}R. Bartnik, \textit{The mass of an asymptotically flat manifold}, Comm. Pure Appl. Math. 39 (1986), 661--693.
\bibitem{Behrndt}T. Behrndt, \textit{Mean curvature flow of Lagrangian submanifolds with isolated conical singularities}, arXiv:1107.4803 (2011).
\bibitem{BehrndtDPhil}T. Behrndt, \textit{Generalized Lagrangian mean curvature flow in almost Calabi-Yau manifolds}, DPhil thesis, University of Oxford, March 2011.
\bibitem{CSS2}S. Coriasco, E. Schrohe, J. Seiler, \textit{Bounded imaginary powers of differential operators on manifolds with conical singularities}, Math. Z. 244 (2003), 235--269.
\bibitem{Davies}E. B. Davies, \textit{Heat kernels and spectral theory}, Cambridge Tracts in Mathematics, 92. Cambridge University Press, Cambridge, 1989.
\bibitem{Joyce}D. D. Joyce, \textit{Special Lagrangian submanifolds with isolated conical singularities. I. Regularity}, Ann. Global Anal. Geom. 25 (2004), 201--251.
\bibitem{KrylovSobolev}N. V. Krylov, \textit{Lectures on elliptic and parabolic equations in Sobolev spaces}, Graduate Studies in Mathematics, 96. American Mathematical Society, Providence, RI, 2008.
\bibitem{KrylovHoelder}N. V. Krylov, \textit{Lectures on elliptic and parabolic equations in H\"older spaces}, Graduate Studies in Mathematics, 12. American Mathematical Society, Providence, RI, 1996.
\bibitem{LM}R. B. Lockhart, R. C. McOwen, \textit{Elliptic Differential Operators on Noncompact Manifolds}, Ann. Scuola Norm. Sup. Pisa Cl. Sci. (4) 12 (1985), 409--447.
\bibitem{Marshall}S. P. Marshall, \textit{Deformations of special Lagrangian submanifolds}, DPhil thesis, University of Oxford, 2002.
\bibitem{Melrose1}R. B. Melrose, \textit{The Atiyah--Patodi--Singer index theorem}, Research Notes in Mathematics, 4. A K Peters, Ltd., Wellesley, MA, 1993.
\bibitem{Mooers}E. Mooers, \textit{Heat kernel asymptotics on manifolds with conic singularities}, J. Anal. Math. 78 (1999), 1--36.
\bibitem{Nagase}M. Nagase, \textit{The fundamental solutions of the heat equations on Riemannian spaces with cone-like singular points}, Kodai Math. J. 7 (1984), 382--455.
\bibitem{Schulze}B.-W. Schulze, \textit{Boundary Value Problems and Singular Pseudo-Differential Operators}, Pure and Applied Mathematics (New York).
John Wiley \& Sons, Ltd., Chichester, 1998.
\bibitem{Yosida}K. Yosida, \textit{Functional Analysis}, Classics in Mathematics. Springer-Verlag, Berlin, 1995.
\end{thebibliography}
\end{document}